\documentclass{amsart}

\usepackage{color}
\usepackage[all]{xy}
\usepackage{amssymb,latexsym,amscd,array,amsmath}
\usepackage[colorlinks,plainpages,backref]{hyperref}
\usepackage{mathrsfs}
\usepackage[neveradjust]{paralist}
\usepackage{epic,eepic}
\usepackage{url,epsfig,psfrag,graphicx}
\usepackage{float,wrapfig}
\usepackage{enumitem}
\usepackage{calc}
\usepackage{bbm}

\def\comment#1{}

\theoremstyle{definition}
\newtheorem{ntn}{Notation}[section]
\newtheorem*{notn}{Notation}
\newtheorem{dfn}[ntn]{Definition}

\theoremstyle{plain}
\newtheorem{lem}[ntn]{Lemma}
\newtheorem{prp}[ntn]{Proposition}
\newtheorem{thm}[ntn]{Theorem}
\newtheorem{cor}[ntn]{Corollary}

\theoremstyle{remark}

\newtheorem{exa}[ntn]{Example}

\def\schluss{\hfill\ensuremath{\diamond}}

\newcommand{\mbc}{\mathbb{C}}

\newcommand{\mbn}{\mathbb{N}}

\newcommand{\ra}{\rightarrow}
\newcommand{\lra}{\longrightarrow}

\newcommand{\mco}{\mathcal{O}}

\newcommand{\mcv}{\mathcal{V}}

\DeclareMathOperator{\Spec}{\textup{Spec}\,}

\DeclareMathOperator{\sHom}{\mathcal{H}\mathit{om}}

\DeclareMathOperator{\pr}{\mathit{p}}

\DeclareMathOperator{\FL}{\textup{FL}}
\DeclareMathOperator{\id}{\textup{id}}

\newcommand{\boldzero}{{\mathbf{0}}}
\newcommand{\boldone}{{\mathbf{1}}}
\newcommand{\bolda}{{\mathbf{a}}}
\newcommand{\boldb}{{\mathbf{b}}}

\newcommand{\boldu}{{\mathbf{u}}}
\newcommand{\boldv}{{\mathbf{v}}}

\newcommand{\del}{\partial}

\newcommand{\eps}{\varepsilon}

\newcommand{\into}{\hookrightarrow}

\renewcommand{\to}{\longrightarrow}
\newcommand{\Mod}{{\mathrm{Mod}}}

\newcommand{\rmR}{{\mathrm{R}}}

\newcommand{\calB}{\mathcal{B}}
\newcommand{\calD}{\mathcal{D}}
\newcommand{\calE}{{\mathcal{E}}}

\newcommand{\calH}{\mathcal{H}}

\newcommand{\calK}{\mathcal{K}}
\newcommand{\calL}{\mathcal{L}}
\newcommand{\calM}{\mathcal{M}}
\newcommand{\calN}{\mathcal{N}}
\newcommand{\calO}{\mathcal{O}}

\newcommand{\calV}{\mathcal{V}}

\newcommand{\rmD}{\mathrm{D}}

\newcommand{\frakH}{\mathfrak{H}}

\renewcommand{\AA}{\mathbb{A}}
\newcommand{\CC}{\mathbb{C}}
\newcommand{\DD}{\mathbb{D}}

\newcommand{\GG}{\mathbb{G}}

\newcommand{\NN}{\mathbb{N}}

\newcommand{\RR}{\mathbb{R}}
\newcommand{\ZZ}{\mathbb{Z}}

\newcommand{\kk}{\mathbbm{k}}

\renewcommand{\bar}{\overline}

\DeclareMathOperator{\Ext}{Ext}

\DeclareMathOperator{\Hom}{Hom}
\DeclareMathOperator{\qdeg}{qdeg}
\DeclareMathOperator{\rk}{rk}

\DeclareMathOperator{\sRes}{sRes}

\DeclareMathOperator{\Tor}{Tor}

\DeclareMathOperator{\Var}{Var}

\numberwithin{equation}{subsection}

\def\cprime{$'$}

\begin{document}
\title[Gau\ss--Manin systems of Laurent polynomials]{Gau\ss--Manin
  systems of families of Laurent polynomials and $A$-hypergeometric
  systems}
%%%%%%%%%%%%%%%%%%%%%%%%%%%%%%%%%%%%%%%%%%%%%%%%%%%%%%%%%%%%%%%%%%%%%%%%%%%%%
%%%%%%%%%%%%%%%%%%%%%%%%%%%%%%%%%%%%%%%%%%%%%%%%%%%%%%%%%%%%%%%%%%%%%%%%%%%%%

\author{Thomas Reichelt}
\address{
T.~Reichelt\\
Mathematisches Institut \\
Universit\"at Heidelberg\\
Im Neuenheimer Feld 205\\
69120 Heidelberg\\
Germany}
\email{treichelt@mathi.uni-heidelberg.de}

\author{Uli Walther}
\address{ U.~Walther\\
  Purdue University\\
  Dept.\ of Mathematics\\
  150 N.\ University St.\\
  West Lafayette, IN 47907\\ USA}
\email{walther@math.purdue.edu}
%%%%%%%%%%%%%%%%%%%%%%%%%%%%%%%%%%%%%%%%%%%%%%%%%%%%%%%%%%%%%%%%%%%%%%%%%%%%%
%%%%%%%%%%%%%%%%%%%%%%%%%%%%%%%%%%%%%%%%%%%%%%%%%%%%%%%%%%%%%%%%%%%%%%%%%%%%%

\begin{abstract}In this note we study families of Gau\ss--Manin systems
  arising from Laurent polynomials with parametric coefficients under
  projection to the parameter space. For suitable matrices of exponent
  vectors, we exhibit a natural four-term exact sequence for which we
  then give an interpretation via generalized $A$-hypergeometric
  systems. We determine the extension groups from the parameter sheaf
  to the middle term of this sequence and show that the four-term
  sequence does not split. Auxiliary results include the computation
  of Ext and Tor groups of $A$-hypergeometric systems against the
  parameter sheaf.
\end{abstract}

\keywords{Gauss--Manin, toric, hypergeometric,
    Euler--Koszul, D-module, Laurent polynomial, equivariant, extension}

\subjclass{13N10,14M25,32C38,32S40,33C70}
% 13N10 Rings of differential operators and their modules
% 14M25 Toric varieties, Newton polyhedra
% 32S40 Monodromy; relations with differential equations and $D$-modules
% 32C38 Sheaves of differential operators and their modules, $D$-modules
% 33C70 Other hypergeometric functions and integrals in several variables

\thanks{TR was supported by a DFG Emmy-Noether-Fellowship (RE 3567/1-1),
%TR and CS acknowledge partial support by the ANR/DFG joint program
  %SISYPH (ANR-13-IS01-0001-01 and SE~1115/5-1).
  UW was supported by the
  NSF under grant 1401392-DMS.
  %All three authors would like to thank the Forschungsinstitut
  %Oberwolfach for hosting them in April of 2015.
  }
\maketitle

\tableofcontents

%Conventions: sheaves cal, modules italic, functors and categories
%roman. Homology italic.

%%%%%%%%%%%%%%%%%%%%%%%%%%%%%%%%%%%%%%%%%%%%%%%%%%%%%%%%%%%%%%%%%%%%%%%%%%%%
\section*{Acknowledgememnts}
%%%%%%%%%%%%%%%%%%%%%%%%%%%%%%%%%%%%%%%%%%%%%%%%%%%%%%%%%%%%%%%%%%%%%%%%%%%%
This article grew out of a discussion with Duco van Straten who raised
the question of the possible splitting of the sequence
\eqref{eq:Pshortexseq}, see Theorem \ref{thm-calcExt}. We would like
to express our sincere thanks to him and to Christian Sevenheck for
generously sharing their ideas.

%%%%%%%%%%%%%%%%%%%%%%%%%%%%%%%%%%%%%%%%%%%%%%%%%%%%%%%%%%%%%%%%%%%%%%%%%%%%
\section{Introduction}
%%%%%%%%%%%%%%%%%%%%%%%%%%%%%%%%%%%%%%%%%%%%%%%%%%%%%%%%%%%%%%%%%%%%%%%%%%%%

During the 1980s {Gel{\cprime}fand, Graev, Kapranov and
  Zelevinski\u\i\ introduced a class of systems of complex partial
  differential equations which are a vast generalization of the
  Gau\ss\ hypergeometric equation and which are nowadays known as
  \emph{$A$-hypergeometric} (or GKZ) \emph{systems} (\emph{cf.}
  \cite{GGZ,GKZ} and a string of other articles of that period). Such
  an $A$-hypergeometric system has a hybrid combinatorial and algebraic
  flavor, its initial datum being an integer matrix $A$ and a
  parameter vector $\beta$ in the column space of $A$. This determines
  a left ideal $H_A(\beta)$ in the Weyl algebra $D$ and the
  $A$-hypergeometric system with respect to $A$ and $\beta$ is then
  the cyclic left $D$-module $M_A^\beta := D/H_A(\beta)$. From this
  definition it is far from clear that these systems have any
  geometric interpretation.

The analytic behavior of $M_A^\beta$ (as a system of PDEs) is highly
dependent on the parameter vector $\beta$. A technique to study this
dependence, the \emph{Euler--Koszul functor}, was developed by
Matusevich, Miller and the second author in \cite{MMW05}.  This is a
functor from the category of toric modules, which are a mild
generalization of $\ZZ A$-graded $\CC[\NN A]$ modules, to the category
of complexes of $D$-modules.  The construction of this functor
generalizes the Euler--Koszul complex on the semi-group ring $\CC[\NN
  A]$ (already known to {Gel{\cprime}fand, Kapranov and
  Zelevinski\u\i, \cite{GKZ}) and was inspired by it.

The Euler Koszul complex provides a $D$-resolution of the
corresponding $A$-hypergeometric system provided $\beta$ does not lie
in the \emph{$A$-exceptional locus}, defined via the local cohomology
of $\CC[\NN A]$. An important step in the geometric
interpretation of $A$-hypergeometric systems was achieved by Schulze
and the second author in \cite{SchulzeWalther-ekdi}, generalizing work
of Gel{\cprime}fand et al. There they
showed, using the Euler--Koszul complex, that the Fourier--Laplace
transform of $M_A^\beta$ can be identifed with the direct image of
a twisted structure sheaf on a torus under a monomial map (depending
on $A$) to affine space whenever $\beta$ is outside the set of \emph{strongly
resonant parameters}.

If $A$ is homogeneous, \emph{i.e.}\ if (1,...,1) is in its row span,
then this embedding descends to an embedding of a torus of dimension
one less into projective space. It was realized by Brylinski
\cite{Brylinski} that the Fourier--Laplace transform of a $D$-module
on affine space which is constant on all punctured lines through the
origin can be expressed by a Radon transform of the corresponding
$D$-module on the projective space.

Using this Radon transform the first author showed in \cite{Reich2}
that homogeneous $A$-hypergeometric systems with not strongly
resonant, but integer parameter vector $\beta$ carry the structure of a
mixed Hodge module. Furthermore, there exists a morphism to $M_A^\beta$
from the Gau\ss--Manin system of the maximal
family of Laurent polynomials with Newton polytope equal to the convex
hull of the columns of $A$. This map 
has $\mco$-free
kernel and cokernel,
and is
compatible with the natural mixed Hodge module structure on the
Gau\ss--Manin system and on $M_A^\beta$ respectively.

Since $M_A^\beta$ is the terminal Euler--Koszul homology of the
semigroup ring $\mbc[\mbn A]$ one wonders whether the Euler--Koszul
homology of other toric modules (for example, $A$-graded ideals of
$\CC[\NN A]$) carry a natural mixed Hodge module structure as well.
In this paper we consider the maximal graded ideal of $\mbc[\mbn A]$
and prove that its terminal Euler--Koszul homology is isomorphic to the
Gau\ss--Manin system of a map whose fibers are the complement of the
fibers of the Laurent polynomial alluded to above.

%The natural inclusion of the graded maximal ideal  into the
%semi-group ring $\mbc[\mbn A]$ gives rise to a long exact sequence of
%Euler--Koszul homology modules ... 

We now give a short overview of the content of this article. In the
first section we review the definition of $A$-hypergeometric systems,
of the Euler--Koszul complex, and several functors on $D$-modules. In
the following section we compute the restriction of $M_A^\beta$ to the
origin, its de Rham cohomology, and the groups
$\Ext^\bullet_D(\calM_A^\beta,\mco)$. A novel feature of this article
is that we work throughout over any field of characteristic zero,
rather than specifically over $\mbc$. In this more general setting we
(re)prove that for not strongly resonant parameter the
Fourier--Laplace transformed $A$-hypergeometric system can be viewed
as the direct image of a twisted structure sheaf under a torus
embedding. In the third section we show that the long exact Euler--Koszul
homology sequences induced by the inclusion of the maximal graded
ideal in $\mbc[\mbn A]$ is isomorphic to certain Gau\ss--Manin systems
coming from a family of Laurent polynomials and compare this sequence
with the sequence obtained in \cite{Reich2}.

\subsection{$A$-hypergeometric systems}
%%%%%%%%%%%%%%%%%%%%%%%%%%%%%%%%%%%%%%%

We introduce here the main notation and review some basis facts on
$A$-hypergeometric systems and the Euler--Koszul functor. We refer to
\cite{MMW05,SchulzeWalther-ekdi} for more details.
\begin{ntn}
  Throughout, we work over the field $\kk$ of characteristic zero.

  In general we adopt the convention that we denote a sheaf by a
  calligraphic letter such as $\calM$, a module by an Italic letter such
  as $M$, and categories and functors by Roman letters such as $\textrm{M}$.
\end{ntn}

\begin{ntn}\label{ntn-A}
Throughout, $A$ will be an integer matrix that we assume to be
pointed: there should be a $\ZZ$-linear functional on the column space of
$A$ that evaluates positively on each column of $A$.

For any integer matrix $A$, let $\kk^A$ be a vector space with basis
corresponding to the columns $\{\bolda_j\}_j$ of $A$. Let $R_A$
(resp.\ $O_A$) be the polynomial ring over $\kk$ generated by the
variables $\del_A=\{\del_j\}_j$ (resp.\ $x_A=\{x_j\}_j$) corresponding
to $\{\bolda_j\}_j$; we read $R_A$ as coordinate ring on the variety
$X_A:=\kk^A$.  Further, let $D_A$ be the ring of $\kk$-linear differential
operators on $O_A$, where we identify $\frac{\del}{\del x_j}$ with
$\del_j$ so that both $R_A$ and $O_A$ are subrings of $D_A$.

For any (semi)ring of coefficients $C$ we write $CA$ for the set
of $C$-linear combinations of the columns of $A$. In particular,
$\kk A$ is a vector space.
\end{ntn}

\begin{dfn}\label{dfn-A}
Let $A$ be an integer matrix with independent rows whose $\ZZ$-ideal of
maximal minors equals $\ZZ$.

For the parameter $\beta\in\kk A$ let $H_A(\beta)$ be the
$D_A$-ideal generated by the \emph{homogeneity equations} 
\begin{eqnarray*}
\{E_i\bullet\phi
  &=&\beta_i\cdot\phi\}_i
\end{eqnarray*}
together with the \emph{toric} partial differential \emph{equations}
\begin{eqnarray*}
\{(\del^{\boldv_+}_A-\del^{\boldv_-}_A)\bullet \phi=0&\mid&
A\cdot \boldv=0\},
\end{eqnarray*}
using (throughout) multi-index notation.
Here, with $\boldzero_A=(0,\ldots,0)$ in $\kk^A$, we write $E_i:=\sum_ja_{ij}x_j\del_j$ and $\boldv_+=\max(\boldv,\boldzero_A)$,
$\boldv_-=-\min(\boldv,\boldzero_A)$.
We put
\[
M_A^\beta:=D_A/H_A(\beta).
\]
\schluss
\end{dfn}

We have
\begin{eqnarray*}
    x^\boldu E_i-E_ix^\boldu  &=& -(A\cdot \boldu)_ix^\boldu,\\
    \del^\boldu E_i-E_i\del^\boldu &=& (A\cdot \boldu)_i\del^\boldu.
\end{eqnarray*}
The \emph{$A$-degree function} (with values in $\ZZ A$) on $R_A$ and
$D_A$ is:
\[
-\deg_A(x_j):=\bolda_j=:\deg_A(\del_j). 
\]
We denote $\deg_{A,i}(-)$ the degree function associated to the weight
given by the $i$-th row of $A$.  Then $E_iP=P(E_i-\deg_{A,i}(P))P$ for
$A$-graded $P$.
Let
\[
\eps_A:=\sum_j\deg_A(\del_j)=\sum_j\bolda_j.
\]

Let $M$ be an $A$-graded $D_A$-module. 
There are
commuting $D_A$-linear endomorphisms $E_i$ via
\[
E_i\circ m:=(E_i+\deg_i(m))\cdot m.
\]
for $A$-graded $m\in M$.  In particular, if $N$ is an $A$-graded
$R_A$-module one obtains commuting sets of $D_A$-endomorphisms on the
left $D_A$-module $D_A\otimes_{R_A}N$ by
\[
E_i\circ (P\otimes Q):=(E_i+\deg_i(P)+\deg_i(Q))P\otimes Q. 
\]

The \emph{Euler--Koszul complex} $K_\bullet(N;E-\beta)$ of the
$A$-graded module $N$ is the homological Koszul complex induced by
$E-\beta:=\{(E_i-\beta_i)\circ\}_i$ on $D_A\otimes_{R_A}N$.  In
particular, the terminal module $D_A\otimes_{R_A}N$ sits in
cohomological degree zero. We denote $\calK_\bullet(N;E-\beta)$ the
corresponding complex of quasi-coherent sheaves. The cohomology
objects are $H_\bullet(N;E-\beta)$ and $\calH_\bullet(N;E-\beta)$
respectively.
If $N(\alpha)$
denotes the usual shift-of-degree functor on the category of
graded $R_A$-modules, then $K_\bullet(N;E-\beta)(\alpha)$ and
$K_\bullet(N(\alpha);E-\beta+\alpha)$ are identical.

Identifying $\ZZ A$ with $\ZZ^{\rk(A)}$ we get coordinates $\{t_i\}_i$
on $T_A=\Spec(\kk[\ZZ A])=\Spec(\kk[\{t_i^\pm\}_i])$ and then an embedding
\begin{gather}
  h_A\colon T_A\to \Spec(\CC[\{\del_j\}_j])=\kk^A
\end{gather}
induced by
the monomial morphism}
\begin{gather}
  \label{eq-monomial-map}
t:=\{t_i\}_i\to \{\prod_i t_i^{\bolda_{ij}}\}_j=:t^A
\end{gather}
The closure of the image of $h_A$ in
$X_A$ becomes a toric variety via $h_A$ and is defined by the
$R_A$-ideal $I_A$ given as the kernel of
\eqref{eq-monomial-map} and generated by all binomials
$\del_A^{\boldv_+}-\del_A^{\boldv_-}$ where $A\boldv=0$. We denote the
semigroup ring 
\[
S_A:=R_A/I_A\simeq \kk[\NN A].
\]
We denote $\widetilde{\NN A}$ the saturation of $\NN A$ and by $\tilde
S_A$ the associated semigroup ring, identical with the normalization of
$S_A$.

The \emph{faces} $\tau$ of the rational polyhedral cone $\RR_+A$,
\emph{i.e.} the subsets of (the columns of) $A$ that minimize (over
$A$) some linear functional $\ZZ A\to \ZZ$, correspond to $A$-graded
prime ideals $I_A^\tau$ of $R_A$ with
$I_A^\tau=I_A+R_A(\{\del_j\}_{j\not\in \tau})$. We let $R_A\del_A$ be
the unique $A$-graded maximal $R_A$-ideal.

An $R_A$-module $N$ is \emph{toric} if it is $A$-graded, and if it has
a (finite) $A$-graded composition chain 
\[
0=N_0\subsetneq N_1\subseteq N_2\cdots\subsetneq N_k=N
\]
such that each composition factor $N_i/N_{i-1}$ is isomorphic as
$A$-graded $R_A$-module to a face ring $R_A/I_A^\tau$ or one
of its shifts by an element of $\ZZ A$. 

For a finitely generated $A$-graded $R_A$-module
$N=\bigoplus_{\alpha\in\ZZ A}
N_\alpha$, let
\begin{eqnarray*}
\deg_A(N)&=&\{\alpha\in\ZZ A\mid N_\alpha\not=0\},\\
\qdeg_A(N)&=&\bar{\deg_A(N)}^{Zar},
\end{eqnarray*}
the latter being the Zariski closure of the former in
$\kk A=\kk\otimes_\ZZ\ZZ A$. For unions of such modules, degrees as
well as quasi-degrees are defined to be the corresponding unions, compare
\cite{SchulzeWalther-ekdi}.

Let $N=\kk(-\alpha)$ be the graded $R_A$-module whose module structure
is that of $R_A/I_A^\emptyset=R_A/R_A\del_A\simeq \kk$, and
which lives entirely inside degree $\alpha\in\ZZ^d$.
Then $K_\bullet(N;E-\beta)$ is an exact complex if $\beta\neq
\alpha$, and its differentials are zero otherwise. 

\begin{dfn}
If the row span of $A$ contains $\boldone_A$ we call $A$
\emph{homogeneous}. Homogeneity is equivalent to $I_A$ defining a
projective variety, and to the system $H_A(\beta)$ having only regular
singularities \cite{SchulzeWalther-Duke}.
\schluss\end{dfn}

\subsection{$\calD$-module functors}
%%%%%%%%%%%%%%%%%%%%%%%%%%%%%%%%%%%%%

Let $X$ be a smooth algebraic $\kk$-variety of dimension $d_X$. We
denote by $\calD_X$ the sheaf of algebraic differential operators and
by $D_X$ its ring of global sections. For $X = \AA^n$ we sometimes
write $D_n$ .  We denote by $\Mod(\calD_{X})$ the Abelian category of
left $\calD_{X}$-modules. The full triangulated subcategories of the
derived category $\rmD^b(\calD_X):=\rmD^b(\Mod(\calD_X))$ consisting of
objects with $\calO_{X}$-quasi-coherent (resp.\ holonomic) cohomology
are denoted by $\rmD^b_{qc}(\calD_X)$ (resp. $\rmD^b_{h}(\calD_X))$.

We recall the notation for cohomological shifting a complex
$C^\bullet$: $C^\bullet[1]$ is the complex $C^\bullet$ shifted one
step left, $(C^\bullet[1])^i=C^{i+1}$, with corresponding shift of the
morphisms. 

Let $f: X \ra Y$ be a map between smooth algebraic
varieties.
Let $\calM \in \rmD^b_{qc}(\calD_X)$ and $\calN \in
\rmD^b_{qc}(\calD_Y)$, then we denote by
\[
f_+ \calM := \rmR f_*(\calD_{Y \leftarrow X} \overset{L}{\otimes} \calM)\quad
\text{ and }\quad f^+ \calN:= \calD_{X \ra Y} \overset{L}{\otimes}
f^{-1}\calN [d_X - d_Y]
\]
the direct and inverse image functors for $\calD$-modules; both
preserve holonomicity and if $f$ is non-characteristic with respect to $\calN$ then $f^+$ is exact (up to a shift),
(see e.g. \cite[Def. 2.4.2 \& Thm 2.4.6]{HottaTakeuchiTanasaki}).
\noindent We denote by 
\begin{align}
\DD: \rmD^b_h(\calD_X) &\lra (\rmD^b_h(\calD_X))^{opp} \notag \\
\calM &\mapsto \rmR\sHom(\calM,\calD_X \otimes_{\calO_X} \Omega^{\otimes-1}_X)[d_X]\notag
\end{align}
the duality functor, which also preserves holonomicity.
We additionally define the functors
\[
f_\dag := \DD \circ f_+
\circ \DD \quad\text{ and }\quad f^\dag := \DD \circ f^+ \circ \DD.
\]

If $X$ is an affine variety, we have an equivalence of categories
\begin{align}
\Mod(\calD_X) &\lra \Mod(D_X) \notag \\
\calM &\mapsto M:= \Gamma(X,\calM)
\end{align}
where $\Mod(D_X)$ is the category of left $D_X$-modules.

\begin{dfn}\label{dfn-FL}
  Let
  \[
  \langle -,-\rangle: \AA^{\ell} \times \hat{\AA}^{\ell} \ra \AA^1,  \qquad
(\lambda_1,\ldots, \lambda_\ell, \mu_1,\ldots \mu_\ell)
  \mapsto \sum_{i=1}^\ell \lambda_i \mu_i.
  \]
  (Here, and elsewhere, $\hat{\AA}^\ell$ denotes an affine space of
  dimension $\ell$; we use the ``hat'' to keep apart source and range
  of the two functors defined in \eqref{eq-FL-dfn} below).  
  Now define two $\calD_{\AA^\ell\times\hat{\AA}^\ell}$-modules by
  \[
  \calL:= \calO_{\AA^\ell \times \hat{\AA}^\ell} e^{\langle\cdot ,\cdot
    \rangle},\qquad \overline{\calL}:= \calO_{\AA^\ell \times \hat{\AA}^\ell}
  e^{-\langle\cdot ,\cdot \rangle}.
  \]
We refer to \cite[Section 5]{KaSc} for details on these sheaves.
  Denote by $p_1: \AA^\ell \times \hat{\AA}^\ell \ra \AA^\ell$ and $p_2: \AA^\ell \times \hat{\AA}^\ell \ra \hat{\AA}^\ell$ the
  projection to the first and second factors respectively. The
  \emph{Fourier--Laplace transform} is defined by
\begin{align}\label{eq-FL-dfn}
\FL: \rmD^b_{qc}(\calD_{\AA^\ell}) &\lra \rmD^b_{qc}(\calD_{\hat{\AA}^\ell}) \notag \\
M &\mapsto p_{2+} (p_1^+ M \overset{L}\otimes \calL)
\end{align}
and
\begin{align}
\FL^{-1}: \rmD^b_{qc}(\calD_{\hat{\AA}^\ell}) &\lra \rmD^b_{qc}(\calD_{\AA^\ell}) \notag \\
M &\mapsto p_{1+} (p_2^+ M \overset{L}\otimes \overline{\calL})
\end{align}
Then $ \FL^{-1} \circ \FL(M) \simeq \iota^+ M$ where $\iota$ is given
by $\lambda \mapsto -\lambda$, and we set
\[
\hat{\calM}^\beta_A := \FL^{-1} (\calM^\beta_A)
\]
with global sections $\hat M_A^\beta$.
\schluss\end{dfn}
\begin{ntn}\label{ntn-FL-stuff}
  If $\AA^\ell$ and $\hat{\AA}^\ell$ are an $\FL$-pair with
  $\AA^\ell=\kk^A$ for some matrix $A$, we shall denote by $\hat R_A,
  \hat O_A, \hat S_A,\ldots$ the $A$-graded objects on
  $\hat{\AA}^\ell$ corresponding to the $A$-graded objects $R_A, O_A,
  S_A,\ldots$ on $\AA^\ell$.
\schluss\end{ntn}

%%%%%%%%%%%%%%%%%%%%%%%%%%%%%%%%%%%%%%%%%%%%%%%%%%%%%%%%%%%%%%%%%%%%%%%%%%%%%
\section{Restriction and de Rham functors of Euler--Koszul complexes}
%%%%%%%%%%%%%%%%%%%%%%%%%%%%%%%%%%%%%%%%%%%%%%%%%%%%%%%%%%%%%%%%%%%%%%%%%%%%%

In this section we make some computations considering certain functors
on the class of (generalized) hypergeometric systems.

\subsection{Local cohomology}
%%%%%%%%%%%%%%%%%%%%%%%%%%%%%%%%%%%%%%%%%%%%%%%%%%%%%%%%%%%%%%%%%%%%%%%%%%%%
Relevant in several ways are the local cohomology functors
$H^\bullet_{\del_A}(-)$ given as the higher derived functors of the
$\del_A$-torsion functor
\[
\Gamma_{\del_A}(M):=\{m\in M\mid
\del_i^k \cdot m=0\, \forall k\gg 0, \forall i\},
\]
a subfunctor of the
identity functor on the category of $R_A$-modules. If $M$ is
$A$-graded, so are all $H^i_{\del_A}(M)$ since the support ideal is
$A$-graded. See \cite{24h} for details and background.

\begin{lem}
\label{lem-lc-tor}
For any $R_A$-module $N$ there is a
functorial isomorphism
\[
\rmR\Gamma_{\del_A}(N)[\dim(X_A)]=(D_A/\del_AD_A)\otimes^L_{R_A}N
\]
so that $H^\bullet_{\del_A}(N)=
\Tor^{R_A}_{\dim(X)-\bullet}(D_A/\del_AD_A,N)$.  Any $R_A$-grading
$\deg(-)$ on $N$ makes this isomorphism graded if the right side is
shifted by $\sum_j\deg(\del_j)$.
\end{lem}
\begin{proof}
One representative for $\rmR\Gamma_{\del_A}(-)$ is the \v Cech
(\emph{i.e.}, stable Koszul) complex ${\check
  C}_A^\bullet(-)=(-)\otimes_{R_A}\bigotimes_j(R_A\to R_A[1/\del_j])$.
On $R_A$, this returns a $D_A$-complex with unique cohomology group,
in cohomological degree $\dim X_A$, given by
$\bigoplus_{\boldv<0}\kk\cdot \del^\boldv$ where $\boldv$ is
componentwise negative. The $D_A$-isomorphism of this module with
$D_A/D_A\del_A$ that identifies the coset of $1/\prod_j \del_j$ in
the former with the coset of $1$ in the latter (is $A$-graded of
degree $\eps_A$ and) shows that (up to this shift in degree) this is
the injective hull of
$R_A/R_A\cdot\del_A$ over $R_A$.
The anti-automorphism
induced by $x^\boldu\del^\boldv\to \del^\boldv(-x)^\boldu$ allows to
view ${\check C}_A^\bullet$ as complex of right $D_A$-modules without
affecting the $R_A$-structure.  Then $H^{\dim X_A}_{\del_A}(R_A)=
D_A/\del_AD_A$ is the canonical module of $R_A$ with its natural right
$D_A$-structure.

The modules in ${\check C}_A^\bullet$ are flat, so ${\check
  C}_A^\bullet\otimes_{R_A}N[\dim X_A]=D_A/\del_A
D_A\otimes_{R_A}^LN$.  If $N$ is graded, then---since $\del_A$ is
monomial---$D_A/\del_AD_A$ and its flat resolution ${\check
  C}_A^\bullet$ are also graded. Hence ${\check
  C}_A^\bullet\otimes_{R_A}N$ has graded cohomology. The
identification $H^{\dim X_A}({\check C}_A^\bullet)[\dim X_A]\simeq
D_A/\del_AD_A$ shifts the grading by the degree of the socle
element $1/\prod_j\del_j$ of the left hand side.
\end{proof}

\subsection{Strongly resonant parameters}
%%%%%%%%%%%%%%%%%%%%%%%%%%%%%%%%%%%%%%%%%%

We recall from \cite{MMW05,SchulzeWalther-ekdi} the following important
sets. The \emph{exceptional locus} $\calE_A$ is 
\[
\calE_A:=\qdeg_A\left(\bigoplus_{k>\dim X_A-\dim
  T_A}\Ext^k_{R_A}(S_A,R_A)\right)=\bigcup_{k<\dim T_A} \left(\bar{\deg_A
H^k_{\del_A}(S_A)}^{Zar}\right).
\]
A larger
interesting set is 
\[
\sRes(A):=\bigcup_j\qdeg_A(H^1_{\del_j}(S_A)), 
\]
the \emph{strongly resonant} parameters of $A$.

For $\kk=\CC$ the following results were shown in
\cite{MMW05,SchulzeWalther-ekdi}.  A parameter is
in $\calE_A$ if and only if the complex $\calK_\bullet(S_A;E-\beta)$
fails to be a resolution of $\calM_A^\beta$; it is in $\sRes(A)$ if
and only if $\calK_\bullet(S_A;E-\beta)$ fails to resolve the
Fourier--Laplace transform of $h_{A+}(\calO_{T_A}^\beta)$ where
\[
\calO_{T_A}^\beta=\calD_{T_A}/\calD_{T_A}(\{\del_{t_i}t_i+\beta_i\}_i),
\]
or
alternatively if and only if $h_{A+}(\calO_{T_A}^\beta)$ disagrees with
$\hat\calM_A^\beta$.  We are interested in these results over $\kk$:

\begin{thm}
Let $\kk$ be an arbitrary field of characteristic zero. For each $j$,
the following are equivalent:
\begin{enumerate}
\item $\beta\not\in\sRes_j(A):=\qdeg_A(H^1_{\del_j}(S_A))$;
\item left-multiplication by $\del_{x_j}$ is a quasi-isomorphism on
  $K_\bullet(E-\beta;S_A)$. 
\end{enumerate}
\end{thm}

\begin{cor}\label{cor:SW}
Over any coefficient field $\kk$ of characteristic zero, the following
are equivalent:
\begin{enumerate}
\item $\beta\not\in\sRes(A)$;
\item $\calK_\bullet(E-\beta;S_A)$ represents the Fourier--Laplace transform of
  $h_{A+}\calO_{T_A}^\beta$%$\phi_+\calM(\beta)$;
\item $\calM_A^\beta$ is naturally isomorphic to the Fourier--Laplace transform
  of $\calH^0 h_{A+}\calO_{T_A}^\beta$ %$\phi_+\calM(\beta)$.
\end{enumerate}
\end{cor}

Inspection shows that, apart from formal computations that do not
depend on $\kk$, there are the following logical dependencies in
\cite{SchulzeWalther-ekdi}.
\begin{itemize}
\item \cite[Cor.~3.7]{SchulzeWalther-ekdi} needs \cite[Thm.~3.5,
  Cor.~3.1,Prop.~2.1]{SchulzeWalther-ekdi} and
  \cite[Prop.~5.3]{MMW05}, and the fact that higher Euler--Koszul
  homology is $(\prod_j\del_j)$-torsion.
\item \cite[Thm.~3.5]{SchulzeWalther-ekdi} needs
  \cite[Lem.~3.2]{SchulzeWalther-ekdi} and
  \cite[Prop.~5.3]{MMW05}. 
\item  \cite[Lem.~3.2]{SchulzeWalther-ekdi} is completely formal and
  independent of the field $\kk$.
\item \cite[Cor.~3.1]{SchulzeWalther-ekdi} needs the left and right
  \O re properties of $D_{A}$, and
  \cite[Prop.~2.1]{SchulzeWalther-ekdi}.
\item \cite[Prop.~2.1]{SchulzeWalther-ekdi} needs that direct images
  over $\kk$ are formally the same for all $\kk$ (which they are),
  plus $\calD$-affinity of tori, plus various formal computations
  contained in \cite{Borel}, namely an identification of a direct
  image module in VI.7.3, the chain rule in VI.4.1, exactness of
  direct images for affine closed embeddings in VI.8.1, and equality of direct
  images under open embeddings in the $\calD$- and $\calO$-category in
  VI.5.2.
\end{itemize}

Tori are $\calD$-affine since they are $\calO$-affine. Higher
Euler--Koszul homology is $(\prod_j\del_j)$-torsion since localizing
every $\del_{j}$ leads to Euler--Koszul homology of the quasi-toric
Cohen--Macaulay module $\kk[\ZZ A]$ (compare
\cite{SchulzeWalther-ekdi} for quasi-toricity). The \O re properties
of $D_A$ rely on the Leibniz rule and are unaffected by $\kk$. Any
closed embedding over $\kk$ can be base-changed to a closed embedding
(and hence to an affine faithful map) over $\CC$, by viewing $\kk$
(algebraically) as a subfield of $\CC$. Since $\CC$ is fully faithful
over $\kk$ and affine faithful maps over $\CC$ yield exact direct
image functors, so do they over $\kk$. Direct images for open
embeddings agree over $\calD$ and $\calO$ more or less by definition.
It therefore remains to inspect \cite[Prop.~5.3]{MMW05} and exactness
of the Euler--Koszul complex on maximal Cohen--Macaulay input over
$\kk$.

Using superscripts to indicate base fields, 
$K_A^\kk(D_A^\kk;E-\beta)\otimes_\kk\CC = K_A^\CC(D_A^\CC;E-\beta)$ as
long as $\beta\in\kk A$. The notion of a toric module
is formally independent of $\kk$, and so the categories of toric modules
and their Euler--Koszul complexes embed into one another for
containments of fields. In particular, the formal mechanisms are
identical and scale from one field to another faithfully.

The required part of \cite[Prop.~5.3]{MMW05} is the equivalence
(3)$\Leftrightarrow$(4). The proof passes through the equivalences
(2)$\Leftrightarrow$(3) and (2)$\Leftrightarrow$(4). For both we
need, modulo formal computations involving toric composition chains,
only to check the equivalence of conditions (1) and (3) in
\cite[Lem.~4.9]{MMW05}.  The implication (1)$\Rightarrow$(3) is linear
algebra over any field. The reverse follows by contradiction from base
change to $\CC$.

Finally, if $M$ is a maximal Cohen--Macaulay toric module over $\kk$
then vanishing of higher Euler--Koszul homology follows like over
$\CC$ from the spectral sequence \cite[Thm.~6.3]{MMW05} since the
existence of the spectral sequence is abstract homological
nonsense. However, this use of the spectral sequence requires the
concept of holonomicity: one would like to use that Euler--Koszul
homology modules are holonomic and that therefore their duals are modules.

The Euler--Koszul homology modules induce $\calD_A$-modules on affine
space. On that class, (dimension, and hence) holonomicity can be
defined over all fields, via the theory of good filtrations. That
holonomic modules have holonomic modules as their duals was proved by Roos, see
the Bernstein notes \cite[Thm.~3.15]{Bernstein-notes}.

\subsection{Restriction to the origin}
%%%%%%%%%%%%%%%%%%%%%%%%%%%%%%%%%%%%%%%%%%%%%%%%%%%%%%%%%%%%%%%%%%%%%%%%%%

Let $\rho$ be the restriction functor to $\boldzero_A\in\kk^A$,
\[
\rho(-):=(D_A/x_AD_A)\otimes^L_{R_A}(-)
\]
from the category of
($A$-graded) $\calD_A$- or $D_A$-modules to the category of ($A$-graded)
$\kk$-vector spaces. Denote $\rho_k(-)$ its $k$-th homology.

We start with a topological observation derived from
\cite{SchulzeWalther-ekdi}. By $H_{dR}^\bullet(-;\kk)$ we mean the
algebraic de Rham cohomology in the sense of Grothendieck
\cite{Grothendieck-dR}.
\begin{lem}\label{lem-restr-sRes}
If $\beta\not\in\sRes(A)$ then $\rho(\calM_A^\beta)$ is
naturally identified with the homology of the local system to
$\calO_{T_A}^\beta$ on the torus $T_A$:
\[
H_j(\rho(\calM_A^\beta))=\rho_j(\calM_A^\beta)\simeq \left\{\begin{array}{cr}
H_{dR}^j (T_A;\kk)&\text{if } \beta\in\ZZ A\smallsetminus \sRes(A);\\
0&\text{if }\beta\in\kk A\smallsetminus(\ZZ A\cup\sRes(A)).
\end{array}\right.
\]
\end{lem}
\begin{proof}
By \cite{SchulzeWalther-ekdi} and \ref{cor:SW}, if
$\beta\not\in\sRes(A)$ then $\FL^{-1}(\calM_A^\beta)\simeq
h_{A+}(\calO_{T_A}^\beta)$.  Under Fourier--Laplace, restriction
$\rho$ converts to the functor $(D_A/\del_AD_A)\otimes^L_{D_A}(-)$. On
the affine space $X_A=\kk^A$ this is the $D$-module direct image under
the map to a point. Hence with $\beta\not\in\sRes(A)$,
$\rho(\calM_A^\beta)$ represents the direct image of $\calO_{T_A}^\beta$
under projection to a point---in other words, the cohomology of the
local system.
\end{proof}

We next extend this lemma by identifying algebraically
$\rho(M_A^\beta)$ with $\bigwedge(\kk A)$ for non-exceptional
$\beta$. (We view the exterior algebra as an abstract copy of the
cohomology of $T_A$). Note that in this case the Euler--Koszul complex resolves
$\calM_A^\beta$ but is not necessarily a representative for
$\FL^{-1} h_{A+}(\calO_{T_A}^\beta)$. Studying restrictions of Euler--Koszul
complexes turns out to be very down to earth.

\begin{lem}\label{lem-restr-EK}
If $\phi\colon N\to N'$ is an $A$-graded morphism (of degree zero) of
$A$-graded $R_A$-modules then 
\begin{asparaenum}
\item 
the restriction $\rho(\calK_\bullet(N;E-\beta))$ is naturally
$\bigwedge (\kk A)\otimes_\kk N_\beta$, in the sense that
\item 
the induced morphism
$\rho(\calK_\bullet(N;E-\beta))\to \rho(\calK_\bullet(N';E-\beta))$ is
identified with the morphism $\bigwedge(\kk A)\otimes_\kk N_\beta \to
\bigwedge (\kk A)\otimes_\kk N'_\beta$ induced from $\phi_\beta$.
\end{asparaenum}
\end{lem}
\begin{proof}
We extend the domain of the Euler--Koszul functor to modules of the
form $Q\otimes_{R_A}N$ where $N$ is an $A$-graded $R_A$-module and $Q$
a right $A$-graded $D_A$-module by setting $E_i\circ(q\otimes
\nu)=q(E_i+\deg_{A,i}(\nu))\otimes \nu$.

Morally, $E_i\circ(-)$ remains right-multiplication by $E_i$ and
(since multiplications on the left and right commute) one easily
checks that there is an isomorphism of functors
$\rho(\calK_\bullet(-;E-\beta))=\calK_\bullet(\rho(D_A\otimes_{R_A}(-));E-\beta)=
\calK_\bullet((D_A/x_AD_A)\otimes^L_{R_A}(-);E-\beta)$ from the category of
$A$-graded $R_A$-modules  to the category of $A$-graded vector
spaces.

As right $R_A$-module, $(D_A/x_AD_A)\otimes_{R_A}N=N$ for any
$A$-graded $N$. Hence
$(D_A/x_AD_A)\otimes^L_{R_A}(-)=(D_A/x_AD_A)\otimes_{R_A}(-)$.  The
$E_i$-action is then $E_i\circ(1\otimes \nu)=(\deg_{A,i}(\nu))\otimes
\nu$.  In particular, the Euler--Koszul complex of $E-\beta$ on
$(D_A/x_AD_A)\otimes N$ is in degree $\alpha\in\ZZ A$ the Koszul
complex on $N_\alpha$ induced by the numbers
$\{\alpha_i-\beta_i\}_i$. If $\alpha=\beta$ then this Koszul complex
is $\bigwedge(\kk A)\otimes_\kk N_\alpha$ with zero differential. If
$\alpha\not=\beta$ then this Koszul complex is the Koszul complex
(over $\ZZ$) of a set of generators of the unit ideal and hence exact.

The final claim is clear from the construction.
\end{proof}

\begin{cor}
  If $\beta\not\in\calE_A$ then
  \[
\rho_j(\calM_A^\beta)\simeq \left\{\begin{array}{ccc}
H_{dR}^j (T_A;\kk)&\text{ if }& \beta\in\NN A;\\
0    &\text{ if }& \beta\not\in\NN A.\end{array}\right.
\]
\end{cor}
\begin{proof}
If $\beta\not\in\calE_A$ then
$\rho(\calM_A^\beta)=\rho(\calK_\bullet(S_A;E-\beta))$. Now use Lemma
\ref{lem-restr-EK}.
\end{proof}

The natural morphism $\rho(\calK_\bullet(S_A;E-\beta))\to
  \rho(\calM_A^\beta)$ need not be an isomorphism:

\begin{exa}
Let $A=\begin{pmatrix}1&1&1&1\\0&1&3&4\end{pmatrix}$ and take
$\beta=(1,2)$, the only parameter with higher Euler--Koszul homology
for this $A$ (by \cite{SturmfelsTakayama98}). 
The $\kk$-dimension vectors for $\rho(\calM_A^\beta)$ and
$\rho(\calK_\bullet(S_A;E-\beta))$ are $(0,0,1,0,0)$ and
$(0,0,0,0,0)$ respectively.
\schluss\end{exa}

In order to better understand the relationship between the
restrictions of the $A$-hypergeometric system
and the Euler--Koszul complex respectively, we consider the $3$-rd
quadrant spectral sequence
\[
E^2_{-i,-j}=\rho_j(\calH_i(-;E-\beta))\Longrightarrow
(\rho(\calK_\bullet(-;E-\beta)))_{i+j}.
\]
The $k$-th differential is $d_k\colon E^k_{-p,-q}\to
E^k_{-p-k+1,-q+k}$. A toric map $N\to N'$ induces a morphism of
corresponding spectral sequences.

All our experiments indicate that if $\beta\in\NN A$ then
$\rho_j(\calM_A^\beta)=H_{dR}^j(T_A)$, irrespective of exceptionality. While
we cannot show that, we have a one-way estimate:
\begin{lem}
If $\beta\in\NN A$ then there is a natural inclusion $\bigwedge(\kk A)\into
\rho(\calM_A^\beta)$.
\end{lem}
\begin{proof}
Consider the morphism of spectral sequences attached to the inclusion
$S_A\into \tilde S_A$ of $S_A$ into its normalization.
For any $\beta\in\NN A$, the
induced map of abutments ${\rho(\calK_\bullet(S_A;E-\beta))\to
\rho(\calK_\bullet(\tilde S_A;E-\beta))}$  is an isomorphism
by Lemma \ref{lem-restr-EK}. Since $\tilde S_A$ is Cohen--Macaulay, it
has no higher Euler--Koszul homology and so the abutment
$\rho(\calK_\bullet(\tilde S_A;E-\beta))$ is stored in the $i=0$ column
of the $E^2$-term. It follows that the isomorphism on abutments must
be coming from the map of the $i=0$ column, for $k\gg 0$. But
$E^k_{0,-j}$ is a submodule of $E^2_{0,-j}$ for $k\geq 2$. In
particular, $\rho(\calK_\bullet(\tilde S_A;E-\beta))\simeq \bigwedge (\kk A)$
is contained in $\rho(\calK_\bullet(S_A;E-\beta))=\rho(\calM_A^\beta)$.
\end{proof}

\subsection{De Rham cohomology}
%%%%%%%%%%%%%%%%%%%%%%%%%%%%%%%%%%%%%%%%%%%%%%%%%%%%%%%%%%%%%%%%%%%%%%%%%%
We consider now the effect of $(D/\del_AD_A)\otimes_{D_A}^L(-)$ on
$M_A^\beta$ and on the Euler--Koszul complex. This behaves differently
since $(D_A/\del_AD_A)\otimes_{D_A}N$ is not $N$ for most $A$-graded
$R_A$-modules $N$.

\begin{dfn}
  If $\beta$
is in $\deg_A(\rmR\Gamma_{\del_A}(S_A))=\bigcup_{k<\dim
  T_A}\deg_A(H^k_{\del_A}(S_A))\subseteq \calE_A$ it is called
\emph{strongly $A$-exceptional}.
\schluss\end{dfn}

\begin{thm}\label{thm-deRham-EK}
For any $A$-graded $R_A$-module $N$,
$(D/\del_AD_A)\otimes_{D_A}^L K_\bullet(N;E-\beta)$ vanishes
whenever $\beta$ is not an $A$-degree of
$\rmR\Gamma_{\del_A}(N))$. More precisely,
\[
(D/\del_AD_A)\otimes_{D_A}^L K_\bullet(N;E-\beta)
  \simeq 
\left(\bigoplus_i H^i_{\del_A}(N)\right)_\beta\otimes_\kk
\bigwedge(\kk A)[\dim X_A].
\]
As in Lemma \ref{lem-restr-EK}, an $A$-graded map $N\to N'$ induces a
map of de Rham complexes that is identified with
$\left(\rmR\Gamma_{\del_A}(N)\to
\rmR\Gamma_{\del_A}(N')\right)_\beta\otimes_\kk \bigwedge(\kk A)[\dim
  X_A]$.

If $\beta$ is not strongly exceptional (\emph{e.g}, if $S_A$ is Cohen--Macaulay),
then 
\[
\Tor_\bullet^{D_A}(D/\del_AD_A,K_\bullet(S_A;E-\beta))=
H_{dR}^{\bullet+\dim X_A}(T_A;\kk)
\]
if $\beta$ is in $\deg_A(H^{\dim T_A}_{\del_A}(S_A))$ and
zero otherwise.
\end{thm}
\begin{proof}
  As in the proof of Lemma \ref{lem-restr-EK},
  we extend the action of the Euler operators to the
quotient $(D_A/\del_AD_A)\otimes_{R_A}N$ for any $A$-graded $N$. Hence
$(D_A/\del_AD_A)\otimes^L_{D_A}K_\bullet(N;E-\beta)=
K_\bullet((D_A/\del_AD_A)\otimes^L_{R_A}N;E-\beta)$ 
for any $A$-graded $R_A$-module $N$.

Recall that $\eps_A=\sum_j\bolda_j$ and that its components
$\eps_{A,i}$ satisfy $E_i+\eps_{A,i}=\sum_ja_{ij}\del_jx_j$.  Take now
a free $A$-graded $R_A$-resolution $F_\bullet$ for $N$. Then for any
$A$-graded element $P\otimes f\in (D_A/\del_AD_A)\otimes_{R_A}F_k$ the
cosets of $(E_i-\beta_i)\circ(P\otimes f)$, of
$(E_i-\beta_i+\deg_{A,i}(P\otimes f))P\otimes f$ and of
$(-\eps_{A,i}-\beta_i+\deg_{A,i}(P\otimes f))P\otimes f$ coincide. So,
as in Lemma \ref{lem-restr-EK}, the Euler--Koszul complex on
$(D_A/\del_AD_A)\otimes F_\bullet$ is in degree $\alpha$ the Koszul
complex on $D_A/\del_AD_A\otimes F_\bullet$ induced by the numbers
$\{-\eps_{A_i}-\beta_i+\alpha_i\}$. Hence
$K_\bullet((D_A/\del_AD_A)\otimes^L_{R_A}N;E-\beta)$ can only have
cohomology when $(D_A/\del_AD_A)\otimes^L_{R_A}N$ has a cohomology
class in degree $\beta+\eps_A$. By Lemma \ref{lem-lc-tor} this is
equivalent to $\beta$ being the degree of a nonzero cohomology class
in ${\check C}_A^\bullet\otimes_{R_A}N$ which proves the first claim.

If $\rmR\Gamma_{\del_A}(N)$ is non-exact in degree $\beta$ then
$(D_A/\del_AD_A)\otimes^L_{R_A}(K_\bullet(N;E-\beta-\eps_A))$ is
$(H^\bullet_{\del_A}(N))_\beta$ tensored with a Koszul complex (shifted by $\dim X_A$) on
$\dim (T_A)$ maps $\kk\to\kk$ each of which is the zero map. Hence in
this case, the resulting cohomology is
$(H^\bullet_{\del_A}(N))_\beta\otimes_\kk \bigwedge(\kk A)[\dim X_A]$.
The indicated naturality condition is clear from the discussion.

If $\beta$ is not strongly exceptional then
$(\rmR\Gamma_{\del_A}(S_A))_\beta\simeq (H^{\dim
  T_A}_{\del_A}(S_A))_\beta$.  The latter is a subquotient of $\kk[\ZZ
  A]$ and hence its $A$-graded Hilbert function takes values in
$\{0,1\}$. The final claim follows.
\end{proof}

\begin{exa}
Let $A=\begin{pmatrix}1&1&1&1\\0&1&3&4\end{pmatrix}$. Then 
$(H^2_{\del_A}(S_A))_\beta$ is nonzero exactly if $\beta$ is an
interior  lattice
point of $-\RR_+ A$, while $H^1_{\del_A}(S_A)$ is a $1$-dimensional
vector space concentrated in degree $(1,2)$. 
It follows that
$\Tor_\bullet^{D_A}(D/\del_AD_A,K_\bullet(S_A;E-\beta))$ is
$H^{\bullet+4}(T_A;\kk)$ when $\beta$ supports $H^2_{\del_A}(S_A)$; it is
the shifted $H^{\bullet+4}(T_A;\kk)[1] =H^{\bullet+5}(T_A;\kk)$ when $\beta=(1,2)$; it is zero in all
other cases.

In particular, no simple general formula (not appealing to local
cohomology modules) for
$\Tor_\bullet^{D_A}(D/\del_AD_A,K_\bullet(S_A;E-\beta))$ comes to
mind.
\schluss\end{exa}

\begin{cor}\label{cor-TorExt-zero}
$\Tor_\bullet^{D_A}(D/\del_AD_A,M_A^\beta)$ and
  $\Ext^\bullet_{D_A}(O_A,M_A^\beta)$ are nonzero only if
\[  
\beta\in
%\bigcup_{i\le\dim(S_A)}\deg_A(H^{i}_{\del_A}(S_A))=
\calE_A\cup(\deg_A(H^{\dim(S_A)}_{\del_A}(S_A))).
\]
\end{cor}
\begin{proof}
Note first that resolving $O_A$ over $D_A$ and dualizing the
resolution gives a resolution of (a cohomologically shifted)
$D_A/\del_AD_A$, so that the Ext- and Tor-claims are equivalent.

By \cite{MMW05}, the Euler--Koszul complex resolves
$M_A^\beta$ whenever $\beta\not\in\calE_A$.  So, for such $\beta$
not in $\deg_A(H^{\dim(S_A)}_{\del_A}(S_A))$, the
indicated Ext- and Tor-groups vanish by Theorem \ref{thm-deRham-EK}.
\end{proof}

\begin{dfn}
Let $N_A$ be the \emph{interior ideal} of $S_A$, generated by the
monomials whose degrees are in the topological interior of $\RR_+A$.
\schluss\end{dfn}

\begin{cor}\label{cor-deRham-GKZ}
If $S_A$ is normal, then
\begin{align}
\Tor_\bullet^{D_A}(D/\del_AD_A,M_A^\beta)&=
\Tor_\bullet^{D_A}(D/\del_AD_A,K_\bullet(S_A;E-\beta)) \notag \\
&= \left\{\begin{array}{cc}H_{dR}^{\bullet + \dim X_A}(T_A;\kk)&\text{ if }-\beta\in \deg_A(N_A);
\\0&\text{else}.\end{array}\right. \notag
\end{align}
\end{cor}
\begin{proof}
The exceptional locus is here empty. The interior ideal is
the canonical module $\omega_{S_A}$ in the $A$-graded category by
\cite[Cor.~6.3.6]{BrunsHerzog} while also in the $A$-graded category
$\omega_{S_A}=\Ext^{\dim X_A-\dim T_A}_{R_A}(S_A,\omega_{R_A})$,
\cite[Prop.~3.6.12]{BrunsHerzog}. Then graded local duality
\cite[Thm.~3.6.19]{BrunsHerzog} yields that 
$\deg_A(H^{\dim T_A}_{\del_A}(S_A))=-\deg_A(N_A)$. Now use Theorem
\ref{thm-deRham-EK}.
% \uli{I think this all works for CM case, but not  sure about proof.}
\end{proof}

For our applications, it is interesting to know that $\NN A$ does not
meet the strongly $A$-exceptional locus where for $\tau\subseteq A$ a
face we write $\del_\tau$ for $\{\del_j\}_{j\in\tau}$:

\begin{lem}\label{lem-deg-lc}
For any face $\tau$ of $A$, no element of $\NN A$ is a degree of $\rmR
\Gamma_{\del_\tau}(S_A)$.
\end{lem}
\begin{proof}
  For $j\in\tau$, $\bigotimes_\tau(S_A\ra S_A[\del_j^{-1}])\simeq
  (S_A\ra S_A[\del_j^{-1}])\otimes \bigotimes_{\tau\ni j'\not
    =j}(S_A\ra S_A[\del_{j'}^{-1}])$. The corresponding double complex
  spectral sequence starts on the $E^1$-page with modules of the form
  $S_A[(\del_j\cdot\prod_{j'\in\tau'}\del_{j'})^{-1}]/S_A[(\prod_{j'\in\tau'}\del_{j'})^{-1}]$
  for all possible $\tau'\subseteq \tau\smallsetminus \{j\}$.  
  The $\kk$-dimension of $A$-graded localizations of $S_A$ in each $A$-degree
  is zero or one, and $S_A$ is a domain,. So,
  $S_A[(\del_j\cdot\prod_{j'\in\tau'}\del_{j'})^{-1}]/S_A[(\prod_{j'\in\tau'}\del_{j'})^{-1}]$
  is of dimension zero in each degree $\beta\in\NN A$. Hence the same
  holds for the abutment.
\end{proof}

In contrast, elements of $\NN A$, including the origin $0$, can indeed
be \emph{quasi}-degrees of lower local cohomology (and hence exceptional
parameters):
\begin{exa}
Let
$A=\begin{pmatrix}2&1&0&1&0\\0&1&1&0&1\\0&0&0&1&1\end{pmatrix}$. The
exceptional locus is the line $\kk\cdot \bolda_1$.
\schluss\end{exa}

The following corollary will be used in Section \ref{sec-4term}
\begin{cor}
Suppose $S_A$ is Cohen--Macaulay, 
and put $M=H_0(S_A\del_A;E-\beta)$.
Then for $\beta=0$,
\[
\Tor_i^{D_A}(D_A/\del_AD_A,M)=\left\{\begin{array}{ccl}
\kk^{\dim T_A}&\text{ if }&i=\dim X_A;\\
\kk&\text{ if }&i=\dim X_A -1;\\
0&\text{ if }&i<\dim X_A -1;\\\end{array}\right.
\]
while all Tor-groups vanish if $0\neq \beta\in\NN A$.
\end{cor}
\begin{proof}
Consider the toric sequence $0\to S_A\del_A\to S_A\to \kk\to
0$. Cohen--Macaulayness ensures, by \cite{MMW05}, that the
Euler--Koszul functor produces an exact sequence
\begin{gather}\label{eq-4term-toric}
0\to H_1(\kk;E-\beta)\to H_0(S_A\del_A;E-\beta)\to M_A^\beta\to
H_0(\kk;E-\beta)\to 0.
\end{gather}
For $\beta\not =0$, the outer modules are zero. For $\beta=0$, the
right module is $O_A$ and the left is $O_A^{\dim T_A}$. The
claim then follows from Theorem \ref{thm-deRham-EK} and Lemma
\ref{lem-deg-lc}: 
apply $\Tor_\bullet^{D_A}(D_A/\del_AD_A,-)$ to $0\to M/O_A^{\dim T_A}\to
M_A^\beta\to O_A\to 0$ and
$0\to O_A^{\dim
  T_A}\to M\to M/ O_A^{\dim T_A}\to 0$.
\end{proof}

\subsection{Ext and the polynomial solution functor}
%%%%%%%%%%%%%%%%%%%%%%%%%%%%%%%%%%%%%%%%%%%%%%%%%%%%%%%%%%%%%%%%%%%%%%%%%%%

Dualizing a $D_A$-resolution of $O_A$ gives a resolution of (a
cohomologically shifted) $D_A/\del_AD_A$. Hence, up to shift by
$\eps_A$ in the $A$-grading,
$\Ext^\bullet_{D_A}(O_A,M_A^\beta)=\Tor^{D_A}_{\dim X_A-\bullet}(D_A/\del_AD_A,M_A^\beta)$.
In particular, the vanishing results in the previous section apply
to $\Ext^\bullet_{D_A}(O_A,M_A^\beta)$.

In this subsection we consider the behavior of the solution functor
$\Hom_{D_A}(-,O_A)$ with values in the ring of polynomials on the
class of $A$-hypergeometric modules $M_A^\beta$. It is immediately
clear that $\Hom_{D_A}(M_A^\beta,O_A)$ can only be nonzero if
$\beta\in\NN A$, and it is an old result that $\beta\in\NN A $ implies
that $\Hom_{D_A}(M_A^\beta,O_A)$ is 1-dimensional, see
\cite[Prop.~3.4.11]{SST}. We investigate here the derived polynomial
solution functor and prove

\begin{thm}\label{thm-ext-gkz-O}
If $\beta\not\in\calE_A$ (for example, if $S_A$ is
Cohen--Macaulay) then
%Let $A\in\ZZ^{d\times n}$ span $\ZZ^d$. Then 
\[
\Ext^i_{D_A}(M_A^\beta,O_A)=\left\{\begin{array}{ccc}
H_{dR}^i(T_A;\kk)&\text{ if }&\beta\in\NN A;\\
0&&else.\end{array}\right.
\]
\end{thm}
(All experiments indicate this
to be true even if $\beta\in\calE_A$.)
\begin{proof}
Write $\tau(-)$ for the transposition $x^\boldu\del^\boldv\mapsto
\del^\boldv (-x)^\boldu$ on $D_A$. Let $F_\bullet$ be an $A$-graded
$R_A$-free resolution of $S_A$ and denote
$\omega_{R_A}=D_A/\del_AD_A[\dim X_A]$.
Then we have the following equalities, where $(-)^\vee$ is the vector
space dual:
\begin{eqnarray*}
\left(\rmR\Hom_{D_A}(M_A^\beta,O_A))\right)^\vee&\stackrel{\text{(a)}}{\simeq}&
\rmR\Hom_{D_A}(O_A,\DD M_A^\beta)\\ 
&\stackrel{\text{(b)}}{\simeq}&\omega_{R_A}\otimes^L_{D_A}\DD M_A^\beta\\ 
&\stackrel{\text{(c)}}{=}&\omega_{R_A}\otimes^L_{D_A}\DD K_\bullet(S_A;E-\beta)\\ 
&\stackrel{\text{(d)}}{=}&\omega_{R_A}\otimes^L_{D_A}\DD K_\bullet(F_\bullet;
E-\beta)\\ 
&\stackrel{\text{(e)}}{=}&\omega_{R_A}\otimes^L_{D_A} K_\bullet(\Hom_{R_A}(F_\bullet,R_A);E+\beta+\eps_A)\\
&\stackrel{\text{}}{\simeq}&\omega_{R_A}\otimes^L_{D_A} K_\bullet(\rmR\Hom_{R_A}(S_A,R_A);E+\beta+\eps_A)\\
&\stackrel{\text{(f)}}{\simeq}&K_\bullet(\omega_{R_A}\otimes^L_{R_A}\rmR\Hom_{R_A}(S_A,R_A);E+\beta+\eps_A)\\
&\stackrel{\text{(g)}}{\simeq}&(\rmR\Gamma_{\del_A}\rmR\Hom_{R_A}(S_A,R_A))_{-\beta-\eps_A}\otimes\bigwedge
(\kk A)\\
&\stackrel{\text{(h)}}{\simeq}&\left((\rmR\Hom_{R_A}(\rmR\Hom_{R_A}(S_A,R_A),R_A))_{\beta}\otimes
\bigwedge (\kk A)\right)^\vee\\
&=&\left((S_A)_\beta\otimes
\bigwedge (\kk A)\right)^\vee.
\end{eqnarray*}
The following notes justify the above transformations:
\begin{enumerate}[label=$\bullet$(\alph*)]
\item Duality gives
  $\rmR\Hom_{D_A}(M,M')\simeq \left(\rmR\Hom_{D_A}(\DD M',\DD M)\right)^\vee$, \cite[\S
  2.6]{HottaTakeuchiTanasaki}.
\item Resolve $O_A$ and dualize the resolution, incurring a
  cohomological shift.
\item By \cite{MMW05}, the hypergeometric
  system is resolved by the Euler--Koszul complex as long as $\beta$
  is not exceptional. 
\item The Euler--Koszul functor can be applied to any $A$-graded
  resolution.
\item $K_\bullet(F_\bullet;E-\beta)$ is a free complex. 
Applying $\Hom_{D_A}(-,D_A)$ and the transposition $\tau$ turns
$D_A\otimes_{R_A}F_\bullet$ into
$D_A\otimes_{R_A}\Hom_{R_A}(F_\bullet,R_A)$ and the Euler--Koszul
complex on $\beta$ into that on $-\beta-\eps_A$ since $x_j\del_j$ turns
into $-\del_jx_j$.
\item As in the proof of Theorem \ref{thm-deRham-EK}.
\item Theorem \ref{thm-deRham-EK} works for $A$-graded complexes just
  as well.
\item Apply local $A$-graded duality (responsible for the dual).
\end{enumerate}
\end{proof}

%%%%%%%%%%%%%%%%%%%%%%%%%%%%%%%%%%%%%%%%%%%%%%%%%%%%%%%%%%%%%%%%%%%%%%%%%%%
\section{Three four-term sequences}\label{sec-4term}
%%%%%%%%%%%%%%%%%%%%%%%%%%%%%%%%%%%%%%%%%%%%%%%%%%%%%%%%%%%%%%%%%%%%%%%%%%%

\begin{notn}
  From now on, $A$ is a $(d+1)\times (n+1)$ matrix
  and $\NN A$ is assumed to be saturated, in addition to the
  conventions in Notation \ref{ntn-A} and Definition \ref{dfn-A}.

All products of $\kk$-schemes are by default over $\kk$.
\schluss\end{notn}

Consider the exact toric sequence $0\to S_A\del_A\to S_A\to \kk\to
0$. Normality ensures, by \cite[Prop.~5.3, Thm.~6.6]{MMW05},
that the Euler--Koszul functor produces the exact sequence
\eqref{eq-4term-toric},
and, for $i\geq 1$, isomorphisms
\begin{gather}\label{eq-EK-kk}
H_i(\kk;E-\beta) \simeq \begin{cases}O_A^{\binom{d+1}{i}} & \text{for}
  \; \beta = 0; \\ 0 & \text{else.} \end{cases}
\end{gather}

In this section we will show that the sequence \eqref{eq-4term-toric}
has a geometric interpretation when $A$ is homogeneous. Our approach
is inspired by \cite{Stienstra98}, where Stienstra defined on the
torus $T_A$ a family $F$ of Laurent polynomials using the matrix
$A$. He showed that one term in the long exact cohomology sequence of
the pair ($T_A$, fiber of $F$) could be naturally identified with a
fiber in the $A$-hypergeometric system $M_A^0$ when $F$ is smooth. We
will extend this identification to the non-smooth fibers of $F$.

We will proceed as follows. First we identify the second term of
\eqref{eq-4term-toric} as a concatenation of (proper) direct image
functors applied to the structure sheaf $\calO_{T_A}$. The third term
already has such an interpretation by 
Corollary \ref{cor:SW} above. The remaining terms are
identified as the cohomology of the cone of a natural adjunction
morphism between the second and third term.

As a second step we show in Lemma \ref{lem:FLsidelift} that the
sequence \eqref{eq-4term-toric} is part of a long exact sequence
coming from a triangle of elementary $\calD$-modules on the line
$\hat{\AA}^1$.  We also show in Proposition \ref{prp-FL-qplus} that
the Fourier--Laplace transform of \eqref{eq-4term-toric} is induced by
the FL-transformed triangle of elementary $\calD$-modules from Lemma
\ref{lem:FLsidelift}.  This enables us to give a geometric
interpretation of the exact sequence in Theorem \ref{thm-GM-toric} in
terms of Gau\ss--Manin systems of the pair ($T_A$, fiber of $F$) as
alluded to above.

As a preparatory result we begin with an identification of two
functors on certain sheaves.

\subsection{Quasi-equivariant bundles}
%%%%%%%%%%%%%%%%%%%%%%%%%%%%%%%%%%%%%%%

Denote $\GG_m$ the scheme of units of $\kk$.  A
\emph{$\GG_m$-action on the variety $Y$} is a multiplicative morphism
$\mu\colon \GG_m\times Y\to Y$ where $1\in\GG_m$ acts as
identity. That is, $\mu$ is a morphism, $\mu(g,\mu(g',y))=\mu(gg',y)$
and $\mu(1,y)=y$.

Let $X$ be an affine smooth variety and $\pi: E = X \times \AA^n \ra
X$ be a trivial vector bundle on $X$. Write $E^\ast = E \setminus (X
\times \{0\})$ and let $E_x$ be the fiber over $x\in X$.
The zero section is identified with $X$ as closed subvariety via the
embedding
\[
i\colon X\into E.
\]
\begin{dfn}
  A $\GG_m$-action $\mu\colon\GG_m\times E\to E$ on $E$ is \emph{fibered} if
  \begin{asparaenum}
  \item $\mu$ preserves fibers, $\mu\colon \GG_m\times E_x\to E_x$;
    %$\mu(\GG_m,E_x)=E_x$;
    \item $\mu$ is the restriction of a morphism $\mu\colon \AA^1\times E\to E$
      under  $\GG_m\into \AA^1$;
    \item $0\in\AA^1$ multiplies into the zero section, $\mu\colon
      0\times E_x\to i(X)$;% $\mu(0,e)=\pi(e)$;
    \item $\AA^1$ fixes the zero section, $\mu\colon \AA^1\times
      i(X)\to i(X)$.% $\mu(\AA^1,\pi(e))=\pi(e)$.
  \end{asparaenum}
\schluss\end{dfn}

\begin{dfn}
  Let $\mu: \GG_m \times E \ra E$ be a fibered $\GG_m$-action on $E$.
  Write $\pr\colon \GG_m\times E\to E$ for the projection and denote
  by $\mu'$ and $\pr'$ the restrictions of $\mu$ and $\pr$ to $\GG_m
  \times E^*$.

  A holonomic $\calD_{E}$-module $\calM$ is called
  \emph{$\GG_m$-quasi-equivariant} if $(\mu')^+ \calM_{\mid E^*}
  \simeq (\pr')^+ \calM_{\mid E^*}$.
\schluss\end{dfn}
We consider the derived category of bounded complexes of
$\calD_E$-modules with holonomic and quasi-equivariant cohomology.

\begin{lem}\label{lem:equivpoint}
Let $\pi: E \ra X$ be fibered and denote $i: X \ra E$ the inclusion
of the zero section.  For every $\GG_m$-quasi-equivariant $\calD_E$-module
$\calM$,
\begin{align}
\pi_+ \calM \simeq i^\dag \calM \qquad \text{and} \qquad \pi_\dag \calM \simeq i^+ \calM \notag
\end{align}
\end{lem}
\begin{proof}
By duality, it suffices to prove the first claim.
Denote by $j: E^\ast \ra E$ the open embedding of the complement of
the zero section and let $\pi$ be the projection to the base
$X$.
%Recall that $a$ is the map to a point.
We
have the exact triangles
\begin{gather}
j_\dag j^{-1} \calM \lra \calM \lra i_+ i^\dag \calM
\overset{+1}{\lra}\\
\label{eqn-pi+j!1/j}
\pi_+ j_\dag j^{-1} \calM \lra \pi_+\calM \lra i^\dag \calM \overset{+1}{\lra}
\end{gather}
and the Cartesian diagram
\[
\xymatrix{\GG_m \times E^* \ar[r]^{j'} \ar[d]_{\mu'} & \AA^1_m \times E
  \ar[d]_\mu \\ E^* \ar[r]^j & E}
\]
where $\mu'$ is the restriction of $\mu$ to $\GG_m \times E^*$ and
$j'$ is the canonical inclusion. The morphism $s: E \ra \AA^1
\times E$ with $s(x) = (1,x)$ is a section of
$\mu$. Thus, the composition (induced by the natural
transformation $\id_E\to\mu_+\mu^\dagger$)
\[
\pi_+ j_\dag j^{-1} \calM \ra \pi_+ \mu_+ \mu^\dag j_\dag j^{-1} \calM   \ra \pi_+ \mu_+ s_+ s^\dag \mu^\dag j_\dag j^{-1} \calM =  \pi_+ j_\dag j^{-1} \calM
\]
is an isomorphism by \eqref{eqn-pi+j!1/j}; hence it is enough to prove
$\pi_+ \mu_+ \mu^\dag j_\dag j^{-1} \calM = 0$. By base change,
\[
\pi_+ \mu_+ \mu^\dag j_\dag j^{-1} \calM \simeq \pi_+ \mu_+ j'_\dag (\mu')^\dag j^{-1} \calM.
\]

Since $\calM$ is $\GG_m$-quasi-equivariant, we have
\[
(\mu')^\dag j^{-1}\calM \simeq \pr^\dag j^{-1}\calM \simeq
\calO_{\GG_m} \boxtimes j^{-1}\calM\, .
\]
Therefore (letting $a\colon \AA^1 \ra
\{pt\}$ be the map to the point) we get
\begin{align*}
\pi_+ \mu_+ j'_\dag(\mu')^\dag j^{-1} \calM &\simeq \pi_+ \mu_+ j'_\dag(\calO_{\GG_m} \boxtimes j^{-1}\calM)  \\
&\simeq \pi_+ \mu_+ (j_{1\dag} \calO_{\GG_m} \boxtimes j_{\dag}j^{-1}\calM) \\
&\simeq \pi_+ \pr_+(j_{1\dag} \calO_{\GG_m} \boxtimes j_{\dag}j^{-1}\calM) \\
&\simeq a_+ j_{1\dag} \calO_{\GG_m} \boxtimes \pi_+ j_{\dag}j^{-1}\calM
\end{align*}
where $j_1 :\GG_m\ra \AA^1$ is the canonical inclusion.
Since $a_+ j_{1\dag} \calO_{\GG_m} =
0$ in $\rmD^b(\calD_{\{pt\}})$ we have $\pi_+ \mu_+\mu^\dag j_\dag j^{-1}
\calM = 0$. 
\end{proof}

Recall that $A$ is a $(d+1)\times(n+1)$ matrix. 
Let $T_A = \Spec(\kk[t_0^\pm,\ldots,t_d^\pm])$ and consider
the ring homomorphism
\begin{align}
\kk[y_0,\ldots,y_n] &\lra \kk[t_0^{\pm},\ldots,t_d^{\pm}] \notag \\
y_i &\mapsto t^{\bolda_i}  \notag
\end{align}
which gives rise to a morphism
\[
h_A: T_A \ra
\hat{\AA}^{n+1},
\]
where $\hat{\AA}^{n+1} = \Spec(\kk[y_0,\ldots,y_n])$.
We factorize this embedding as
\[
T_A \overset{h_1}{\lra} \hat{\AA}^{n+1} \setminus \{0\} \overset{h_2}{\lra} \hat{\AA}^{n+1}\, .
\]
We are now ready to show a useful property of $A$-hypergeometric
systems and their Fourier--Laplace transforms.
\begin{lem}\label{lem:hotequiv}
The $\calD_{\hat{\AA}^{n+1}}$-module $h_{A+} \calO_{T_A}$ is
$\GG_m$-quasi-equivariant.
\end{lem}
\begin{proof}
We view ${\hat A}^{n+1}$ as trivial bundle over itself.
  Since $A$ is pointed, there is $\boldu\in\ZZ^{d+1}$ with
  $\boldv=\boldu^T\cdot A$ componentwise positive.
  Let $\mu'\colon \GG_m\times
  \hat{\AA}^{n+1}\to\hat{\AA}^{n+1}$ be the monomial action induced by
  $\boldv$ and let $\tilde\mu\colon \GG_m\times T_A\to T_A$ be
  the action induced by $\boldu$. (Compare the discussion on the Euler
  space in \cite{RSW}.) 
Consider the Cartesian diagram
\[
\xymatrix{ \GG_m\times T_A \ar[r]^-{h_1\times id}
  \ar@<0.5ex>[d]^{\tilde{\mu}} \ar@<-0.5ex>[d]_{\pr} &
  \GG_m\times \hat{\AA}^{n+1}\setminus \{0\} \ar@<0.5ex>[d]^{\mu'}
  \ar@<-0.5ex>[d]_{\pr} & \\ T_A \ar[r]^-{h_1} &
  \hat{\AA}^{n+1} \setminus \{0\} \ar[r]^-{h_2} & \hat{\AA}^{n+1}}
\]
and note that the positivity of $\boldv$ allows to extend $\mu'$ to
$\AA^1\times \hat{\AA}^{n+1}$.
Then 
\[
\pr^+ h_{1+} \calO_{T_A} \simeq (h_1 \times id)_+ \pr^+
\calO_{T_A} \simeq (h_1 \times id)_+ \tilde{\mu}^+
\calO_{T_A} \simeq (\mu')^+ h_{1+} \calO_{T_A}
\]
and so $\pr^+ h_{1+}\calO_{T_A} \simeq (\mu')^+
h_{1+}\calO_{T_A}$.
\end{proof}

\subsection{The four-term sequence in terms of direct images}
%%%%%%%%%%%%%%%%%%%%%%%%%%%%%%%%%%%%%%%%%%%%%%%%%%%%%%%%%%%%%%%%%
Recall Notation \ref{ntn-FL-stuff} regarding Fourier--Laplace
transforms and consider the inverse Fourier--Laplace transformation of
the sequence \eqref{eq-4term-toric}:
\begin{gather}\label{eq-4term-toric-FL}
0 \lra H_1(\kk;\hat{E}+\beta) \lra H_0(\hat{S}_A \cdot
y_A;\hat{E}+\beta) \lra \hat{M}^\beta_A \lra H_0(\kk;\hat{E}+\beta)
\lra 0.
\end{gather}

\begin{dfn}
  Let $\calB_{0}$ be the unique simple
  $\calD_{\hat{\AA}^{n+1}}$-module supported in
  $0\in\hat{\AA}^{n+1}$.
\schluss\end{dfn}

\begin{prp}\label{prop:FL4termexseq}
For $\beta=0$ there is an isomorphism of exact 4-term sequences
\[
\xymatrix{0 \ar[r] & \calH_1(\kk;\hat{E}) \ar[r] & \calH_0(
  \hat{S}_A \cdot y_A;\hat E) \ar[r] & \hat{\calM}^0_A \ar[r] &
  \calH_0(\kk;\hat{E}) \ar[r] & 0 \\ 0 \ar[r]& {\calB_0}^{d+1}
  \ar[u]^\simeq \ar[r] & \calH^0(h_{2\dag} h_{1+}
  \calO_{T_A}) \ar[u]^\simeq \ar[r] &
  \calH^0(h_{+}\calO_{T_A}) \ar[u]^\simeq \ar[r] & \calB_0
  \ar[r] \ar[u]^\simeq & 0}
\]
\end{prp}
\begin{proof}
By Corollary \ref{cor:SW}, $h_{A+} \calO_{T_A}
\simeq \hat{\calM}^0_A$. By \eqref{eq-EK-kk},
$\calH_i(\kk;\hat E)\simeq
{\calB_0}^{{d+1\choose i}}$.
Restricting to $\hat{\AA}^{n+1} \setminus
\{0\}$ we see that
\[
\calH_0(\hat{S}_A \cdot y_A;\hat E)_{\mid \hat{\AA}^{n+1} \setminus
  \{0\}} \simeq (\hat{\calM}^0_A)_{\mid \hat{\AA}^{n+1} \setminus
  \{0\}} \simeq h_{1+} \calO_{T_A} \qquad \text{in}\quad
\Mod_h(\calD_{\hat{\AA}^{n+1}})\, .
\]
Since $\calH_{>0}(S_A \cdot y_A;\hat E)_{\mid \hat{\AA}^{n+1}
  \setminus \{0\}} = 0$, we have
\[
\calH_0(S_A \cdot y_A;\hat E)_{\mid \hat{\AA}^{n+1} \setminus \{0\}}
\simeq \calK_\bullet(S_A \cdot y_A;\hat E)_{\mid \hat{\AA}^{n+1}
  \setminus \{0\}} \qquad \text{in} \quad
\rmD^b_h(\calD_{\hat{\AA}^{n+1}})\, .
\]
By adjunction this gives a morphism  
\[ 
h_{2 \dag}h_{1+} \calO_{T_A} \overset{\simeq}\lra h_{2 \dag}
h_{2}^{-1}\calK_\bullet(\hat{S}_A \cdot y_A;\hat E) \lra
\calK_\bullet(\hat{S}_A \cdot y_A;\hat E)
\]
and so induces a morphism
$\calH^0 (h_{2 \dag}h_{1+} \calO_{T_A}) \lra
\calH_0(\hat{S}_A \cdot y_A;\hat E)$
such that the center and right squares in our diagram commute. We
need to prove that the morphism
\begin{equation}\label{eq:adjEuKos}
h_{2 \dag} h_2^{-1} \calK_\bullet(\hat{S}_A \cdot
y_A;\hat E) \lra \calK_\bullet(\hat{S}_A \cdot y_A;\hat E)
\end{equation}
is an isomorphism. (While we know that $\calH_1(\kk;\hat E)$ and
${\calB_0}^{d+1}$ are isomorphic, it is not yet clear that 
$\calH^0 (h_{2\dag} h_{1+}
\calO_{T_A}) \ra \calH_0( \hat{S}_A \cdot
y_A;\hat E)$ induces such isomorphism.)

In order to prove that the morphism \eqref{eq:adjEuKos} is an
isomorphism we have to show that the third term in the adjunction
triangle
\begin{equation}\label{eq:adjEuKosTriangle}
h_{2 \dag} h_2^{-1} \calK_\bullet(\hat{S}_A \cdot y_A;\hat E) \lra
\calK_\bullet(\hat{S}_A \cdot y_A;\hat E) \lra i_+ i^{\dag}
\calK_\bullet(\hat{S}_A \cdot y_A;\hat E) \overset{+1}\lra
\end{equation}
vanishes. By Kashiwara equivalence it is enough to show that $i^{\dag}
\calK_\bullet(\hat{S}_A \cdot y_A;\hat E)$ is isomorphic to zero.
Since $\calK_\bullet(\widehat{S}_A \cdot y_A;\hat E)_{\mid
  \hat{\AA}^{n+1} \setminus \{0\}} \simeq (h_{A+} \calO_T)_{\mid
  \hat{\AA}^{n+1} \setminus \{0\}}$, the complex $\calK_\bullet(\widehat{S}_A
\cdot y_A;\hat E)$ is $\GG_m$-quasi-equivariant.  By Lemma
\ref{lem:equivpoint}, $i^\dag \calK_\bullet(\widehat{S}_A \cdot
y_A;\hat E) \simeq a_+ \calK_\bullet(\widehat{S}_A \cdot y_A;\hat E)$
where $a$ is the map to a point. Now $a_+ \calK_\bullet(\widehat{S}_A
\cdot y_A;\hat E)$ is dual to $\rho\calK_\bullet(S_A\del_A;E)$ which
allows us to use Lemma \ref{lem-restr-EK} to conclude.
\end{proof}

\subsection{The four-term sequence with Gau\ss--Manin systems}
%%%%%%%%%%%%%%%%%%%%%%%%%%%%%%%%%%%%%%%%%%%%%%%%%%%%%%%%%%%%%

\begin{notn}
  From now on, in addition to the assumptions in Notation \ref{ntn-A}
  and Definition \ref{dfn-A} as
  well as normality, we assume that the matrix $A$ is
  homogeneous, \emph{i.e.} that $(1,\ldots,1)$ is in the row span of $A$.

  Furthermore, for the remainder of the paper, $\beta=0$.
\schluss\end{notn}
One may put $A$ into the following shape by elementary row
operations
\begin{equation}\label{def:Atilde}
A = (\bolda_1, \ldots, \bolda_n) =\left(\begin{matrix}
  1 &  1 & \ldots & 1 \\
  0 &  a_{11} & \dots & a_{1n} \\
  \vdots & \vdots & &\vdots \\
  0 & a_{d1} & \dots &a_{dn}\end{matrix} \right) =
\left(\begin{matrix}
  1 &  1 & \ldots & 1\\
  0 &   &  &  \\
  \vdots & & B & \\
  0 &  &  &\end{matrix} \right)
\end{equation}
where $B= (\boldb_1,\ldots,\boldb_n)$ is the $d \times
n$-matrix with entries $(a_{ij})_{1\leq i \leq d, 1 \leq j\leq n}$.

Using this homogeneity assumption, we will here give a geometric
interpretation to our 4-term sequence \eqref{eq-4term-toric}.  For
this we will need a variant of a comparison theorem of d'Agnolo and
Eastwood \cite{AE}, between the Radon and Fourier--Laplace transform,
and several other preparatory statements.

Set $T_B := \Spec(\kk[t_1^\pm,\ldots,t_d^\pm])$ and $\hat{\AA}^1 :=
\Spec(\kk[t_0])$.  We will identify $T_A$ with $T_B \times (\hat{\AA}^1
\setminus \{0\})$. From the ring homomorphism
\begin{align}
\kk[y_0,\ldots,y_n] &\lra \kk[t_0,t_1^\pm,\ldots,t_d^\pm]\notag
\\ (y_0,\ldots,y_n) &\mapsto (t_0,t_0 t^{\boldb_1},\ldots,t_0
t^{\boldb_n}) \notag
\end{align}
we get a map
\begin{gather}\label{eq-map-k}
k\colon T_B \times \hat{\AA}^1 \lra \hat{\AA}^{n+1}
\end{gather}
whose restriction to $T_A$ is just our
old morphism $h_A$. Let $\tilde{k}$ be the closed embedding
\[
\tilde{k}:= (id_{T_B} \times k) :T_B
\times \hat{\AA}^1 \ra T_B \times \hat{\AA}^{n+1},
\]
let $j,i$ be the embedding and inclusion
\[
j\colon T_A=T_B \times (\hat{\AA}^1
\setminus \{0\})\ra T_B \times \hat{\AA}^1,\qquad
i\colon T_B \times \{0\} \ra T_B \times \hat{\AA}^1.
\]
Then there is a commutative diagram
\[
\xymatrix{\{0\} \ar[d]_{i_0} & \ar[l]_a T_B \ar[d]_i \ar[dr]^{k_0} & \\ 
\hat{\AA}^1  &  T_B \times \hat{\AA}^1 \ar[r]^k \ar[l]_{p} & \hat{\AA}^{n+1} \\ 
\hat{\AA}^1 \setminus\{0\} \ar[u]^{j_0}& T_A \ar[l]_{~p'} \ar[ur]^{h_A} \ar[u]^j \ar[r]^{h_1~} & \hat\AA^{n+1}\setminus \{0\} \ar[u]_{h_2}}
\]
where $p\colon T_B \times \hat{\AA}^1 \ra \hat{\AA}^1$ is the projection and
$k_0$ sends $T_B$ to the origin. Define the following $\calD$-modules on
$\hat{\AA}^1$:
\[
 \calD_{\hat{\AA}^1} \bullet 1 := \calD_{\hat{\AA}^1}/(\del_t), \qquad
 \calD_{\hat{\AA}^1} \bullet \frakH := \calD_{\hat{\AA}^1}/(t \del_t),
 \]
 \[
 \calD_{\hat{\AA}^1} \bullet 1/t := \calD_{\hat{\AA}^1}/(\del_t t),
 \qquad \calD_{\hat{\AA}^1} \bullet \delta := \calD_{\hat{\AA}^1} / (t).
 \]
 (The module $\calD_{\hat{\AA}^1} \bullet \frakH$ encodes the
 Heaviside distribution).
\begin{lem}\label{lem:FLsidelift}%Consider the following commutative diagram

We have the following isomorphisms:
\begin{align}
k_+ \calO_{T_B \times \hat{\AA}^1} &\simeq k_+ p^+(\calD_{\hat{\AA}^1}
\bullet 1), & h_{A+} \calO_{T_A} &\simeq
k_+p^+(\calD_{\hat{\AA}^1} \bullet 1/t), \notag \\ k_{0+}\calO_{T_B} &\simeq
k_+ p^+ (\calD_{\hat{\AA}^1} \bullet \delta), & h_{2 \dag} h_{1+}
\calO_{T_A} &\simeq k_+ p^+(\calD_{\hat{\AA}^1} \bullet \frakH).
\notag
\end{align}
The adjunction morphism $h_{2 \dag} h_{1+} \calO_{T_A} \ra
h_{A+} \calO_{T_A}$ is induced by the adjunction morphism
$\calD_{\hat{\AA}^1}\bullet \frakH \ra \calD_{\hat{\AA}^1}\bullet 1/t$.
\end{lem}
\begin{proof}
The first three isomorphisms follow from
\begin{align}
 k_+ \calO_{T_B \times \hat{\AA}^{1}} &\simeq k_+ p^+
 \calO_{\hat{\AA}^1} = k_+ p^+ (\calD_{\hat{\AA}^1} \bullet 1)\notag,
 \\ h_{A+} \calO_{T_A} &\simeq k_+ j_+ \calO_{T_A}
 \simeq k_+ j_+ {p'}^+ \calO_{\hat{\AA}^1 \setminus \{0\}} \simeq k_+ p^+
 j_{0+}\calO_{\hat{\AA}^1 \setminus \{0\}} \simeq k_+ p^+
 (\calD_{\hat{\AA}^1} \bullet 1/t), \notag \\ k_{0+} \calO_{T_B} &\simeq k_+
 i_+ \calO_{T_B} \simeq k_+ i_+ a^+ \calO_{\{0\}} \simeq k_+ p^+ i_0
 \calO_{\{0\}} \simeq k_+ p^+ (\calD_{\hat{\AA}^1} \bullet \delta).
 \notag
\end{align}
For the last one we have 
\[
k_+ p^+ (\calD_{\hat{\AA}^1} \bullet \frakH) \simeq k_+ p^+ j_{0\dag}
\calO_{\hat{\AA}^1 \setminus \{0\}} \simeq k_+ j_\dag {p'}^+
\calO_{\hat{\AA}^1 \setminus \{0\}} \simeq k_+ j_\dag
\calO_{T_A}.
\]
So it remains to prove that $k_+ j_\dag \calO_{T_A} \simeq
h_{2\dag}h_{1+} \calO_{T_A}$. For this consider the diagram
with Cartesian squares
\[
\xymatrix{T_B \ar[d]^i \ar[r]^{p_0} & \{0\} \ar[d]_{i_0} \\ T_B \times
  \AA^1 \ar@/^/[u]^p\ar[r]^k & \AA^{n+1}
  \ar@/_/[u]_\pi\\ T_A \ar[u]_j \ar[r]^{h_1} &
  \AA^{n+1}\setminus \{0\} \ar[u]_{h_2}}
\]
Base change in the lower square gives 
$h_{1+} \calO_{T_A} \stackrel{\simeq}{\ra}\, h_{1+} j^{-1}j_\dag
\calO_{T_A} \stackrel{\simeq}{\ra} h^{-1}_2 k_+
j_{\dag}\calO_{T_A}$. Adjunction yields a morphism $h_{2
  \dag} h_{1+} \calO_{T_A} \ra k_+ j_{\dag}
\calO_{T_A}$. In order to prove that this is an isomorphism,
it is hence enough to show that $h_{2\dag}h_{1+}\calO_{
  T_A}\simeq h_{2\dag}h_2^{-1} k_+ j_\dag
\calO_{T_A} \ra k_+ j_\dag \calO_{T_A}$ is an
isomorphism. Using the triangle
\[
h_{2 \dag}h_2^{-1} k_+ j_\dag \calO_{T_A} \lra k_+ j_\dag
\calO_{T_A} \lra i_{0+}i_0^\dag k_+ j_\dag
\calO_{T_A} \overset{+1}\lra\,,
\]
it remains to show that $i_{0+}i_0^\dag k_+ j_\dag \calO_{T_A}$ is
zero. For this we observe that 
\[
h_2^+ k_+ j_\dag \calO_{T_A} \simeq h_{1+} j^+ j_\dag
\calO_{T_B} \simeq h_{1+} \calO_{T_B}
\]
is the restriction of a quasi-equivariant module. 
This shows, via Lemma \ref{lem:hotequiv}, that $k_+ j_\dag
\calO_{T_A}$ is $\GG_m$-quasi-equivariant. We therefore have
\[
i_{0+}i_0^\dag k_+ j_\dag \calO_{T_A} \simeq i_{0+} \pi_+
k_+ j_\dag \calO_{T_A} \simeq i_{0+}p_{0+} p_+j_\dag
\calO_{T_A} \simeq i_{0+}p_{0+} i^{\dag} j_\dag \calO_{T_A}\,,
\]
using Lemma \ref{lem:equivpoint} to substitute $p_+$ by $i^!$.
Since $i^\dag j_\dag \calO_{T_A}$ is zero, the claim
follows.
\end{proof}
%We will now compute the convolution of the kernels above with the
%kernel $\calL$ of the Fourier--Laplace transform.
Consider the diagram
\[
\xymatrix{\AA^1 & T_B \times \AA^{n+1} \ar[l]_-F \ar[r]^-q & \AA^{n+1}}
\]
where $q$ is the projection and 
\[
F(t_1,\ldots,t_n,\lambda_0,\ldots,\lambda_n)
= \lambda _0 + \sum_{i=1}^n \lambda_i t^{\boldb_i}.
\]
Denote $\Gamma=\Var(F)$ and write
\[
i_\Gamma\colon \Gamma \subset T_B \times \AA^{n+1},\quad j_U\colon U
\ra T_B \times \AA^{n+1}
\]
for the inclusion of $\Gamma$ and its complement $U$. The
Gau\ss--Manin system $q_{U+}\calO_U$ is of interest since it carries a
mixed Hodge structure by Saito's work in \cite{Saito-MHM-Kyoto90}.
Our article gives evidence to our belief that many $D$-modules
arising from Euler--Koszul complexes also carry such structure, and
that they relate to interesting geometric information.

\begin{prp}\label{prp-FL-qplus}
With $u = 1,\delta, 1/t, \frakH$ and $\hat{u} = \delta,1,\frakH,1/t$,
and with $k$ as in \eqref{eq-map-k} we have the following isomorphisms
\[
\FL(k_+ p^+ (\calD_{\hat{\AA^1}} \bullet \hat{u})) \simeq q_+ F^+
(\calD_{\AA^1} \bullet u). 
\]
\end{prp}
\begin{proof}
Consider the diagram
\[
\xymatrix@C=4pc@R=3pc{ \hat{\AA}^1 & T_B \times \hat{\AA}^1 \ar[l]_{p}
  \ar[r]^{k} & \hat{\AA}^{n+1} \\ \hat{\AA}^1 \times \AA^1
  \ar[u]_{p_1} \ar[d]_{p_2} & T_B \times \hat{\AA}^1 \times \AA^{n+1}
  \ar[u]_{p_{12}} \ar[l]_{id_{\hat{\AA}^1} \times F} \ar[d]^{p_{13}}
  \ar[r]^{k\times id_{\AA^{n+1}}} & \hat{\AA}^{n+1} \times \AA^{n+1}
  \ar[u]_{q_{1}} \ar[d]^{q_{2}} \\ \AA^1 & T_B \times \AA^{n+1} \ar[l]_F
  \ar[r]^q& \AA^{n+1}}
\]
where $p_{ij}$ are the projections to the factors $i$ and $j$.  Recall
the Fourier--Laplace sheaf $\calL$ on $\AA^{n+1}$ from Definition
\ref{dfn-FL} and denote $\calL_1$ the Fourier--Laplace sheaf on
$\AA^1\times\AA^1$.  Then
\begin{align*}
  \FL(k_+p^+ (\calD_{\hat{\AA}^1}\bullet \cdot \hat{u})) &= q_{2+}\left((q_1^+
  k_+ p^+(\calD_{\hat{\AA}^1} \bullet \hat{u})) \otimes
  \calL\right)\\
  &\simeq q_{2+}((k\times id)_+\; p_{12}^+\; p^+
  (\calD_{\hat{\AA}^1} \bullet \hat{u}) \otimes \calL) \\
  &\simeq q_{2+}(k
  \times id)_+(p_{12}^+ p^+(\calD_{\hat{\AA}^1} \bullet \hat{u}) \otimes
  (k \times id)^+\calL)\\
  &\simeq q_+ p_{13+}((id \times F)^+ p_1^+
  (\calD_{\hat{\AA}^1} \bullet \hat{u}) \otimes (k \times
  id)^+\calL)\\
  &\simeq q_+ p_{13+}((id \times F)^+ p_1^+
  (\calD_{\hat{\AA}^1} \bullet \hat{u}) \otimes (id \times
  F)^+\calL_1)\\
  &\simeq q_+ p_{13+}(id \times F)^+( p_1^+
  (\calD_{\hat{\AA}^1} \bullet \hat{u}) \otimes \calL_1) \\
  &\simeq q_+
  F^+ p_{2+} p_1^+ (\calD_{\hat{\AA}^1} \bullet \hat{u}) \otimes
  \calL_1)\\
  &\simeq q_+ F^+ (\calD_{\AA^1} \bullet u).
\end{align*}

\end{proof}
Now consider the diagram
\[
\xymatrix{\{0\} \ar[d]_{i_0} & \Gamma \ar[l]_a \ar[d]^{i_\Gamma}
  \ar[dr]^{q_\Gamma} & \\ \AA^1 & T_B \times \AA^{n+1} \ar[l]_F \ar[r]^q
  & \AA^{n+1} \\ \AA^1 \setminus \{0\} \ar[u]^{j_0} & U \ar[u]_{j_U}
  \ar[l]_<<<<<<{F|_U} \ar[ur]_{q_U} &}
\]
where $q$ is the projection. We have, writing $F$ for $F_U$,
\begin{align*}
q_+ F^+ (\calD_{\AA^1} \bullet 1/t) &\simeq q_+ F^+ j_{0+} \calO_{\AA^1
  \setminus \{0\}} \simeq q_+ j_{U +} F^+ \calO_{\AA^1 \setminus
  \{0\}} \simeq q_{U +} \calO_U, \\
q_+ F^+ (\calD_{\AA^1} \bullet
\frakH) &\simeq q_+ F^+ j_{0\dag} \calO_{\AA^1 \setminus \{0\}} \simeq
q_+ j_{U \dag} F^+ \calO_{\AA^1 \setminus \{0\}} \simeq q_+ j_{U \dag}
\calO_U, \\
q_+F^+ (\calD_{\AA^1}\bullet \delta) &\simeq q_+ F^+ i_{0+}
\calO_{\{0\}}\simeq q_+ i_{\Gamma +}a^+ \calO_{\{0\}} \simeq q_+
i_{\Gamma +} \calO_\Gamma \simeq q_{\Gamma+} \calO_\Gamma, \\
q_+ F^+(\calD_{\AA^1} \bullet 1) &\simeq q_+ \calO_{T_B \times \AA^{n+1}},
\end{align*}
where the second isomorphism in the second line follows from the
smoothness of $F$.

\begin{ntn}
  If $W$ is a $\kk$-space (for example, $H^i_{dR}(T_B;\kk)$) then
  $\underline W$ denotes the trivial vector bundle $W \otimes_\kk
  \calO_{\AA^{n+1}}$.
\schluss\end{ntn}

Consider the following exact sequence of $\calD_{\hat{\AA}^1}$-modules
\begin{align}\label{eq-I}
0 \lra \calD_{\hat{\AA}^1} \bullet \delta \lra \calD_{\hat{\AA}^1} \bullet
\frakH \lra \calD_{\hat{\AA}^1} \bullet 1/t \lra \calD_{\hat{\AA}^1} \bullet
\delta \lra 0
\end{align}

\begin{thm}\label{thm-GM-toric}
The exact sequence \eqref{eq-I} induces an isomorphism of
exact sequences
\[
\xymatrix{0\ar[r] &
  \calH_1(\kk;E)\ar[r] &
  \calH_0(S_A\del_A;E)\ar[r] &
  \calM_A^0\ar[r] &
  \calH_0(\kk;E)\ar[r] & 0 \\
  0 \ar[r]&
  {\begin{array}{c}\underline{H^d_{dR}(T_B;\kk)}\\
      \oplus\\
      \underline{H^{d-1}_{dR}(T_B;\kk)}\end{array}}\ar[r] \ar[u]_\simeq &
  \calH^0 (q_{U+} \calO_U) \ar[r] \ar[u]_\simeq &
  \calH^0 (q_+ j_{U \dag} \calO_U) \ar[r] \ar[u]_\simeq &
  \underline{H^d_{dR}(T_B;\kk)} \ar[u]_\simeq \ar[r] & 0}
\]
\end{thm}
\begin{proof}
The sequence \eqref{eq-I} is part of the long exact sequence coming from
the triangle
\[
j_{0 \dag}j_0^{-1} \calO_{\hat{\AA}^1} \lra j_{0 +} j_0^{-1} \calO_{\hat{\AA}^1} \lra i_{0+} i_0^\dag j_{0 +} j_0^{-1} \calO_{\hat{\AA}^1} \overset{+1}\lra
\]
which is isomorphic to
\[
\calD_{\hat{\AA}^1} \bullet \frakH \lra \calD_{\hat{\AA}^1} \bullet 1/t \lra
(\calD_{\hat{\AA}^1} \bullet \delta) \oplus (\calD_{\hat{\AA}^1} \bullet
\delta[1]) \overset{+1}\lra
\]
Applying the concatenated functor $\textup{FL} \circ k_+ p^+$ to the
triangle above and using Lemma \ref{lem:FLsidelift}, Proposition
\ref{prop:FL4termexseq}, and the fact that $\calH^i(k_{0+}\calO_{T_B})=
{\calB_0}^{\binom{d}{i}}$ we obtain the upper sequence in the
theorem. (Recall that $k_0$ sends $T_B$ to the origin in
$\hat{\AA}^{n+1}$). Applying $q_+ F^+ \circ \FL$ instead gives the lower
sequence.
\end{proof}

If one applies $q_+ F^+ \circ \FL$ to the exact sequence
\begin{align}\label{eq-II}
 0 \lra \calD_{\hat{\AA}^1} \bullet 1 \lra \calD_{\hat{\AA}^1} \bullet 1/t
 \lra \calD_{\hat{\AA}^1} \bullet \delta \lra 0
\end{align}
one obtains as a part of the resulting long exact sequence the piece
\begin{equation}\label{eq:seqII}
0 \lra \underline{H^{d-1}_{dR}(T_B;\kk)} \lra \calH^0(q_{\Gamma+}
\calO_\Gamma) \lra \calH^0(q_+ j_{U \dag} \calO_U) \lra
\underline{H^d_{dR}(T_B;\kk)} \lra 0,
\end{equation}
We  now determine how this sequence relates to the two sequences
in Theorem \ref{thm-GM-toric}.
\begin{prp}
The exact sequence \eqref{eq:seqII}
is the quotient of the exact sequence 
\[
0 \lra \underline{H^{d}_{dR}(T_B;\kk)} \oplus
\underline{H^{d-1}_{dR}(T_B;\kk)} \lra \calH^0(q_{U+} \calO_U) \lra
\calH^0(q_+ j_{U \dag} \calO_U) \lra \underline{H^d_{dR}(T_B;\kk)} \lra 0
\]
by the exact sequence 
\[
0 \lra \underline{H^{d}_{dR}(T_B;\kk)} \lra \underline{H^{d}_{dR}(T_B;\kk)} \lra 0 \lra 0 \lra 0.
\]
\end{prp}

\begin{proof}
Consider the Fourier--Laplace transforms of the sequences \eqref{eq-I}
and \eqref{eq:seqII}. We get a commutative diagram with exact rows and
columns:
\[
\xymatrix{ & 0 \ar[r] & \calD_{\AA^1} \bullet 1 \ar[r] \ar[d] & \calD_{\AA^1} \bullet 1 \ar@{^{(}->}[d] \ar[r] & 0 \ar[r] \ar[d] & 0 \ar[r] \ar[d] & 0 \\ 
 & 0 \ar[r] & \calD_{\AA^1} \bullet 1 \ar[r] \ar[d] &  \calD_{\AA^1}  \bullet 1/t \ar[r] \ar@{->>}[d] & \calD_{\AA^1}  \bullet \frakH \ar[r] \ar[d] & \calD_{\AA^1}  \bullet 1 \ar[r] \ar[d] & 0 \\
 &  0 \ar[r] & 0 \ar[r] &  \calD_{\AA^1}  \bullet \delta \ar[r]  & \calD_{\AA^1}  \bullet \frakH \ar[r]  & \calD_{\AA^1}  \bullet 1 \ar[r]  & 0}
\]
and morphisms of triangles
\[
\xymatrix{ \calD_{\AA^1} \bullet 1 \ar[r] \ar[d] & 0 \ar[r] \ar[d] & (\calD_{\AA^1} \bullet 1) [1] \ar[r]^-{+1} \ar[d] & \\  
\calD_{\AA^1}  \bullet 1/t \ar[r] \ar[d] & \calD_{\AA^1}  \bullet \frakH \ar[r] \ar[d] & (\calD_{\AA^1}\bullet 1) \oplus (\calD_{\AA^1} \bullet 1) [1] \ar[d] \ar[r]^-{+1} & \\
 \calD_{\AA^1}  \bullet \delta \ar[r]  & \calD_{\AA^1}  \bullet \frakH \ar[r]  & \calD_{\AA^1}  \bullet 1 \ar[r]^-{+1}&}
\]
From this, we get an
exact sequence of exact rows
\[
\xymatrix{0 \ar[r] &   \underline{H^{d}_{dR}(T_B;\kk)} \ar[r] \ar@{^{(}->}[d]& \underline{H^{d}_{dR}(T_B;\kk)} \ar[r] \ar@{^{(}->}[d] & 0 \ar[r] \ar[d] & 0 \ar[r] \ar[d] & 0 \\
0 \ar[r]&  \underline{H^{d}_{dR}(T_B;\kk)} \oplus  \underline{H^{d-1}_{dR}(T_B;\kk)} \ar[r] \ar@{->>}[d] & \calH^0(q_{U+} \calO_U)  \ar[r] \ar@{->>}[d] & \calH^0(q_+ j_{U \dag} \calO_U) \ar[r] \ar[d]& \underline{H^d_{dR}(T_B;\kk)} \ar[r] \ar[d]& 0 \\
0 \ar[r] &  \underline{H^{d-1}_{dR}(T_B;\kk)} \ar[r] & \calH^0(q_{\Gamma+} \calO_\Gamma) \ar[r] & \calH^0(q_+ j_{U \dag} \calO_U) \ar[r] & \underline{H^d_{dR}(T_B;\kk)} \ar[r] & 0}
\]
The lower middle maps are surjective since $F^+$ is
exact and $q_+$ is right exact. 
\end{proof}

We now introduce a family of Laurent polynomials defined on $T_B \times
\AA^n$ using the columns of the matrix $B$. For this, recall
Definition 
\ref{def:Atilde} and  consider the ring
homomorphism
\begin{align}
\kk[\lambda_0, \ldots, \lambda_n] &\lra \kk[t_1^\pm, \ldots ,t_d^\pm]
\otimes_\kk \kk[\lambda_1, \ldots, \lambda_n] \notag \\ \lambda_i
&\mapsto \begin{cases} -\sum_{i=1}^n t^{\boldb_i}\otimes \lambda_i &
  \text{for}\; i= 0; \\ \lambda_i & \text{for}\;
  i=1,\ldots,n,\end{cases}
\end{align}
which induces a family of Laurent polynomials
%Consider the following family of Laurent polynomials associated to
%the matrix $B$:
\begin{align}\label{eq:FamLaurent}
\varphi_B\colon T_B \times \AA^n &\lra \AA^{n+1} = \AA^{1} \times \AA^{n} \, .
\end{align}
and an isomorphism
\begin{align}
i_\varphi: T_B \times \AA^n &\lra \Gamma \subseteq T_B \times \AA \times
\AA^n \notag
%\\ (t_1,\ldots,t_d,\lambda_1,\ldots,\lambda_n) &\mapsto
%(t_1,\ldots, t_d, -\sum_{i=1}^n \lambda_i t^{\boldb_i},
%\lambda_1,\ldots, \lambda_n) \notag
\end{align}
onto the graph $\Gamma$. Hence $\varphi_B = q_\Gamma \circ
i_\varphi$ and therefore $\calH^0(\varphi_{B+} \calO_{T_B \times
  \AA^n}) \simeq \calH^0(q_{\Gamma +} \calO_\Gamma)$.

This recovers a special case of a theorem of \cite{Reich2},
\emph{i.e.} there is an exact sequence
\begin{gather}\label{eq:seqIII}
\xymatrix{0 \ar[r] & \underline{H^{d-1}_{dR}(T_B;\kk)} \ar[r] &
  \calH^0(\varphi_{B+} \calO_{T_B \times \AA^n}) \ar[r] & \calM^0_A
  \ar[r] & \underline{H^d_{dR}(T_B;\kk)} \ar[r] & 0}
\end{gather}
which is isomorphic to the sequence \eqref{eq:seqII}.

\subsection{Vanishing Gau\ss--Manin system and the extension class}
%%%%%%%%%%%%%%%%%%%%%%%%%%%%%%%%%%%%%%%%%%%%%%%%%%%%%%%%%%%%%%%%%%%

In this section we show that the $A$-hypergeometric system is an
extension of a trivial vector bundle of rank one
by the quotient of 
a Gau\ss--Manin system modulo
%the $\mcd$-module generated by
its flat sections. We show
that this extension does not split.

As before,
$\beta=0$ (and $A$ is saturated, homogeneous, and pointed).
%, and as in
%Definition \ref{dfn-A}.

\begin{dfn}
  The \emph{vanishing Gau\ss--Manin system} $\calV$ with respect to
  the map $\varphi_B$ is the cokernel of the map
  $\underline{H^{d-1}_{dR}(T_B\kk)} \lra \calH^0(\varphi_{B+}
  \calO_{T_B \times \AA^n})$. In other words,
  \begin{equation}\label{eq:Pshortexseq}
    0 \lra \mcv \lra \calM^0_A \lra \underline{H^d_{dR}(T_B;\kk)} \lra 0
  \end{equation}
  is exact.  We write $V_A=\Gamma(\AA^{n+1},\calV_A)$ and note the
  short exact sequence
  \[
  0\to V_A\to M^0_A\to O_A\to 0.
  \]
\schluss\end{dfn}
The terminology is borrowed from the vanishing cohomology of a
hyperplane section $j:X \hookrightarrow Y$ of an $n$-dimensional
projective variety $Y$ which is a direct summand $H^{n-1}(X) =
H^{n-1}(X)_{van} \oplus j^*H^{n-1}(Y)$.

The sheaf $\calV$ appears perhaps for the first time in Stienstra's
article \cite[Formula (61)]{Stienstra98}, essentially as a restriction
of \eqref{eq:Pshortexseq} to the smooth locus (where all sheaves in
\eqref{eq:seqII} become vector bundles). However, our situation is
more general even in Stienstra's set-up since in \cite{Stienstra98} the matrix
$B$ is assumed to be homogeneous while it is arbitrary for us.

A natural question is: what is the extension class of $\calM^0_A$ inside the
sequence \eqref{eq:Pshortexseq}? Our next result answers this question,
confirming a prediction of Duco van Straten.

\begin{thm}\label{thm-calcExt}
Write $\calO$ for $H^d_{dR}(T_B;\kk)
\otimes \calO_{\AA^{n+1}}$.  There are natural (in $\kk$) isomorphisms 
  \[
  \Ext^i_\calD(\calO,\calV)\simeq \begin{cases}
    \kk &
    \text{for}\; i = 1 \\ 0 & \text{else.} \end{cases} 
  \]
  The class of the sequence \eqref{eq:Pshortexseq} is nonzero and
  induced by the identity on $\calO$ under the connecting morphism.
\end{thm}
\begin{proof}
  Since $\AA^{n+1}$ is affine it suffices to compute on the level of
  global sections.  By Corollary \ref{cor-TorExt-zero},
  $\Ext^\bullet_{D_A}(O_A,M_A^\beta)$ vanishes for $\beta\in \NN A$.
  Hence, $\Ext^i_{D_A}(O_A,V_A)=\Ext^{i-1}_{D_A}(O_A,O_A)$ and so has
  exactly the prescribed $\kk$-space structure. In particular,
  \eqref{eq:Pshortexseq} does not split.
  
  The class of \eqref{eq:Pshortexseq} inside $\Ext^1_{D}(O_A,V)\simeq
  \Ext^{0}_{D}(O_A,O_A)$ is the image of the identity on $\calO_A$
  under the connecting morphism induced by \eqref{eq:Pshortexseq},
  compare \cite[Sec.~3.4]{Weibel}. Since the connecting
  morphism is  an isomorphism, this element is nontrivial.
\end{proof}

\bibliographystyle{amsalpha}
\bibliography{extgkz}

\newcommand{\etalchar}[1]{$^{#1}$}
\def\cprime{$'$}
\providecommand{\bysame}{\leavevmode\hbox to3em{\hrulefill}\thinspace}
\providecommand{\MR}{\relax\ifhmode\unskip\space\fi MR }
\providecommand{\MRhref}[2]{%
  \href{http://www.ams.org/mathscinet-getitem?mr=#1}{#2}
}
\providecommand{\href}[2]{#2}
\begin{thebibliography}{BGK{\etalchar{+}}87}

\bibitem[Ber]{Bernstein-notes}
Joseph Bernstein, \emph{Algebraic theory of {D}-modules.}, Unpublished notes
  available online at {\tt
  http://www.math.uchicago.edu/∼arinkin/langlands/Bernstein/Bernstein-dmod.pdf}.

\bibitem[BGK{\etalchar{+}}87]{Borel}
A.~Borel, P.-P. Grivel, B.~Kaup, A.~Haefliger, B.~Malgrange, and F.~Ehlers,
  \emph{Algebraic {$D$}-modules}, Perspectives in Mathematics, vol.~2, Academic
  Press Inc., Boston, MA, 1987.

\bibitem[BH93]{BrunsHerzog}
Winfried Bruns and J{\"u}rgen Herzog, \emph{Cohen-{M}acaulay rings}, Cambridge
  Studies in Advanced Mathematics, vol.~39, Cambridge University Press,
  Cambridge, 1993. \MR{1251956 (95h:13020)}

\bibitem[Bry86]{Brylinski}
Jean-Luc Brylinski, \emph{Transformations canoniques, dualit\'e projective,
  th\'eorie de {L}efschetz, transformations de {F}ourier et sommes
  trigonom\'etriques}, Ast\'erisque (1986), no.~140-141, 3--134, 251,
  G{\'e}om{\'e}trie et analyse microlocales.

\bibitem[DE03]{AE}
Andrea D'Agnolo and Michael Eastwood, \emph{Radon and {F}ourier transforms for
  {$\mcd$}-modules}, Adv. Math. \textbf{180} (2003), no.~2, 452--485.

\bibitem[GGZ87]{GGZ}
I.~M. Gel$\prime$fand, M.~I. Graev, and A.~V. Zelevinski\u\i, \emph{Holonomic
  systems of equations and series of hypergeometric type}, Dokl. Akad. Nauk
  SSSR \textbf{295} (1987), no.~1, 14--19. \MR{902936}

\bibitem[Gro66]{Grothendieck-dR}
A.~Grothendieck, \emph{On the de {R}ham cohomology of algebraic varieties},
  Inst. Hautes \'Etudes Sci. Publ. Math. (1966), no.~29, 95--103. \MR{0199194}

\bibitem[GZK89]{GKZ}
I.~M. Gel$\prime$fand, A.~V. Zelevinski\u\i, and M.~M. Kapranov,
  \emph{Hypergeometric functions and toric varieties}, Funktsional. Anal. i
  Prilozhen. \textbf{23} (1989), no.~2, 12--26. \MR{1011353}

\bibitem[HTT08]{HottaTakeuchiTanasaki}
Ryoshi Hotta, Kiyoshi Takeuchi, and Toshiyuki Tanisaki, \emph{{$D$}-modules,
  perverse sheaves, and representation theory}, Progress in Mathematics, vol.
  236, Birkh\"auser Boston, Inc., Boston, MA, 2008, Translated from the 1995
  Japanese edition by Takeuchi. \MR{2357361}

\bibitem[ILL{\etalchar{+}}07]{24h}
Srikanth~B. Iyengar, Graham~J. Leuschke, Anton Leykin, Claudia Miller, Ezra
  Miller, Anurag~K. Singh, and Uli Walther, \emph{Twenty-four hours of local
  cohomology}, Graduate Studies in Mathematics, vol.~87, American Mathematical
  Society, Providence, RI, 2007. \MR{2355715 (2009a:13025)}

\bibitem[KS97]{KaSc}
Masaki Kashiwara and Pierre Schapira, \emph{Integral transforms with
  exponential kernels and {L}aplace transform}, J. Amer. Math. Soc. \textbf{10}
  (1997), no.~4, 939--972. \MR{1447834}

\bibitem[MMW05]{MMW05}
Laura~Felicia Matusevich, Ezra Miller, and Uli Walther, \emph{Homological
  methods for hypergeometric families}, J. Amer. Math. Soc. \textbf{18} (2005),
  no.~4, 919--941 (electronic). \MR{2163866 (2007d:13027)}

\bibitem[{Rei}14]{Reich2}
Thomas {Reichelt}, \emph{{Laurent Polynomials, GKZ-hypergeometric Systems and
  Mixed Hodge Modules}}, Compositio Mathematica \textbf{(150)} (2014),
  911--941.

\bibitem[RSW17]{RSW}
Thomas Reichelt, Christian Sevenheck, and Uli Walther, \emph{On the
  $b$-functions of hypergeometric systems}, Int. Math. Res. Not. (to appear
  2017), 1--14.

\bibitem[Sai90]{Saito-MHM-Kyoto90}
Morihiko Saito, \emph{Mixed {H}odge modules}, Publ. Res. Inst. Math. Sci.
  \textbf{26} (1990), no.~2, 221--333. \MR{1047415}

\bibitem[SST00]{SST}
Mutsumi Saito, Bernd Sturmfels, and Nobuki Takayama, \emph{Gr\"obner
  deformations of hypergeometric differential equations}, Algorithms and
  Computation in Mathematics, vol.~6, Springer-Verlag, Berlin, 2000.
  \MR{1734566 (2001i:13036)}

\bibitem[ST98]{SturmfelsTakayama98}
Bernd Sturmfels and Nobuki Takayama, \emph{Gr\"obner bases and hypergeometric
  functions}, Gr\"obner bases and applications ({L}inz, 1998), London Math.
  Soc. Lecture Note Ser., vol. 251, Cambridge Univ. Press, Cambridge, 1998,
  pp.~246--258. \MR{1708882 (2001c:33026)}

\bibitem[Sti98]{Stienstra98}
Jan Stienstra, \emph{Resonant hypergeometric systems and mirror symmetry},
  Integrable systems and algebraic geometry ({K}obe/{K}yoto, 1997), World Sci.
  Publ., River Edge, NJ, 1998, pp.~412--452. \MR{1672077}

\bibitem[SW08]{SchulzeWalther-Duke}
Mathias Schulze and Uli Walther, \emph{Irregularity of hypergeometric systems
  via slopes along coordinate subspaces}, Duke Math. J. \textbf{142} (2008),
  no.~3, 465--509. \MR{2412045 (2009b:13067)}

\bibitem[SW09]{SchulzeWalther-ekdi}
\bysame, \emph{Hypergeometric {D}-modules and twisted {G}au\ss-{M}anin
  systems}, J. Algebra \textbf{322} (2009), no.~9, 3392--3409. \MR{2567427
  (2010m:14028)}

\bibitem[Wei94]{Weibel}
Charles~A. Weibel, \emph{An introduction to homological algebra}, Cambridge
  Studies in Advanced Mathematics, vol.~38, Cambridge University Press,
  Cambridge, 1994. \MR{1269324 (95f:18001)}

\end{thebibliography}

\end{document}